\documentclass{amsart}
\usepackage{amsthm,amscd,amsfonts,amssymb, color}
\usepackage{hyperref}
\numberwithin{equation}{section}
\newtheorem{lemma}[equation]{Lemma}
\newtheorem{proposition}[equation]{Proposition}
\newtheorem{corollary}[equation]{Corollary}
\newtheorem{definition}[equation]{Definition}
\newtheorem{remark}[equation]{Remark}
\newtheorem{conjecture}[equation]{Conjecture}
\newcommand\ve{\varepsilon}
\newcommand\vN{\vec N}
\newcommand\lexp[2]{\kern\scriptspace\vphantom{#2}^{#1}\kern-\scriptspace#2}
\newcommand\hl{{\ell}}
\newcommand\lS{{\ell_{\bS^{\pm1}}}}
\newcommand\lZ{{\ell_{\BZ}}}
\newcommand\inv{^{-1}}
\newcommand\bb{{\mathbf b}}
\newcommand\bg{{\mathbf g}}
\newcommand\bh{{\mathbf h}}
\newcommand\bi{{\mathbf i}}
\newcommand\bj{{\mathbf j}}
\newcommand\bp{{\mathbf p}}
\newcommand\bs{{\mathbf s}}
\newcommand\bt{{\mathbf t}}
\newcommand\bu{{\mathbf u}}
\newcommand\bv{{\mathbf v}}
\newcommand\bw{{\mathbf w}}
\newcommand\bx{{\mathbf x}}
\newcommand\by{{\mathbf y}}
\newcommand\bI{{\mathbf I}}
\newcommand\bJ{{\mathbf J}}
\newcommand\bK{{\mathbf K}}
\newcommand\bN{{\mathbf N}}
\newcommand\bS{{\mathbf S}}
\newcommand\bW{{\mathbf W}}
\newcommand\BC{{\mathbb C}}
\newcommand\BN{{\mathbb N}}
\newcommand\BR{{\mathbb R}}
\newcommand\BZ{{\mathbb Z}}
\newcommand\CC{{\mathcal C}}
\DeclareMathOperator\Id{Id}
\DeclareMathOperator\rightlcm{right-lcm}
\DeclareMathOperator\leftgcd{left-gcd}
\DeclareMathOperator\inverse{inv}
\DeclareMathOperator\pr{pr}
\DeclareMathOperator\hpr{\widehat{pr}}
\DeclareMathOperator\hpi{\hat\pi}
\DeclareMathOperator\bpi{\boldsymbol\pi}
\DeclareMathOperator\product{prod}
\DeclareMathOperator\rev{rev}
\DeclareMathOperator\supp{supp}
\DeclareMathOperator\head{head}
\DeclareMathOperator\tail{tail}
\DeclareMathOperator\conj{Conj}

\makeindex
\begin{document}
\author{Fran\c cois Digne, Eddy Godelle  and Jean Michel}
\title{Retraction to a parabolic subgroup and applications} 

\begin{abstract}
We  continue the study of the retraction  from an Artin group to a standard
parabolic  subgroup introduced  by Blufstein,  Charney, Paris and the second
author.  Using right  and left  retractions we  obtain new  results on minimal
parabolic subgroups, intersection of parabolic subgroups, double cosets with
respect to parabolic subgroups and conjugacy classes in Artin groups.
\end{abstract}
\maketitle
\section{Introduction}

Given a non-empty set \index{S@$S$} $S$, a Coxeter matrix is a symmetric matrix
$(m_{s,t})_{s,t\in  S}$ such that $m_{s,s} = 1$ and $m_{s,t}\in \BN_{\geq
2}\cup \{\infty\}$ for $s\neq t$. To such a Coxeter matrix one can associate a
group, namely its associated Coxeter group, defined by the following
presentation \index{W@$W$}

$$ W = \biggl\langle S \  \biggl | \   \begin{array}{cl} s^2 = 1 &\textrm{
for } s\in S \\ \underbrace{sts\cdots}_{m_{s,t} \textrm{ terms}}  =
\underbrace{tst\cdots}_{m_{s,t} \textrm{ terms}} & \textrm{ for } s\neq t \textrm{ and }
m_{s,t}\neq\infty\end{array}  \biggr\rangle \leqno{(*)} $$  

Given a copy \index{S@$\bS$} $\bS$ of $S$ one can also define the associated Artin group
\index{B@$B$}

$$ B= \biggl\langle \bS\  \biggl | \    \underbrace{\bs\bt\bs\cdots}_{m_{s,t} \textrm{ terms}}  = \underbrace{\bt\bs\bt\cdots}_{m_{s,t} \textrm{ terms}}  \textrm{ for } s\neq t \textrm{ and } m_{s,t}\neq\infty  \biggr\rangle 
\leqno{(**)}$$

We  denote by \index{B@$B^+$} $B^+$ the submonoid of  $B$ generated by $\bS$; its elements are
called positive braids. We denote the length on $W$ with respect to $S$
by $\ell_S : W \to \BN$ \index{lS@$\ell_S$} \index{lZ@$\lZ$}
and we denote by $\lZ: B\to\BZ$ the group morphism which sends $\bS$ to
$1$.
It immediately follows
from the presentations that the map $\bs\in \bS \mapsto s\in S$ induces a
morphism  \index{pr@$\pr$} $\pr: B \to W$. Its kernel is called the
pure Artin group \index{P@$P$} associated to $B$
(or the coloured Artin group) and will be denoted by $P$. Moreover, the inverse
bijection $s\mapsto \bs$ extends by lifting reduced expressions 
to a canonical  section (of sets) from $W$ to $B$.
In the sequel \index{W@$\bW$} $\bW$ will denote  the image of this section. This is a subset
of $B^+$. The elements of $\bW$ are called simple braids. We will use the
following notations: for \index{w@$\bw, w$} $\bw\in\bW$ the notation $w$ means $\pr(\bw)$
and for $w\in W$ the notation $\bw$ means the lift of $w$ to $\bW$. In other words
$\bw$ is the unique element in $B^+$ such that $\pr(\bw) = w$ and $\lZ(\bw)
= \ell_S(w)$. \\

Many papers focus on Artin groups, but whereas Coxeter groups are almost
well-understood,  this is not the case for Artin groups. Even the word problem
remains open in general. The strategy used to prove most of the results
obtained is to use a particular family of subgroups, namely the family of
standard  parabolic subgroups. This is why in \cite{Charney-Paris} the authors
claim that ``any result on these subgroups is likely to be useful''. For $I\subseteq S$,
the standard  parabolic  subgroup  $B_I$  (\emph{resp.}  $W_I$)  of $B$
\index{BI@$B_I$}  \index{W@$W_I$}  (\emph{resp.}  of  $W$)  is the
subgroup generated by $\bI\subseteq\bS$ lifting
$I$ (\emph{resp.} generated  by $I$). Such  a subgroup is an Artin
group  (\emph{resp.} a Coxeter group) for the submatrix $(m_{s,t})_{s,t\in I}$
(\emph{cf.}  \cite{VanderLek}, see also Proposition~\ref{vdl}).
A parabolic subgroup is any subgroup that is
conjugate to a standard parabolic subgroup.

The present article studies the properties of a retraction on standard parabolic
subgroups which has been defined in~\cite{blufstein-paris}.
This is a continuation of a series of articles
\cite{Godelle-Paris,Charney-Paris,blufstein-paris,godelle}  on this  topic.
A  complex $\textrm{Sal}(B)$, called the Salvetti complex, can be associated
with  each Artin group. Furthermore,  for every standard parabolic subgroup
$B_I$  there exists  a canonical  embedding $\textrm{Sal}(B_I) \to
\textrm{Sal}(B)$.  In  \cite{Godelle-Paris},  the  authors  prove that this
embedding  admits a  retraction. Then  in \cite[Theorem 1.2]{Charney-Paris}, the authors
use  the above  retraction to  prove that  standard parabolic  subgroups of
Artin  groups are convex (see also Proposition~\ref{convexity}). 
Along  the way, they use  the above retraction to
construct  a retraction $\bpi_I$  from the Artin  group $B$ to its standard
parabolic subgroup $B_I$.
The construction was implicit in the proof of the main result of
\cite{Charney-Paris}  and  is  made  explicit  in~\cite{blufstein-paris}.
In \cite{godelle}  an alternative construction of this retraction is obtained,
which is based only on algebraic and elementary arguments; several algebraic
properties of the retraction are also obtained.

The present article will
take this approach further. One of the aims is to relate the
definition of the retraction to root systems and make it sufficiently explicit
that it is easy to calculate. The retraction is not a morphism, and
a goal is to  better  understand the retraction of a
product.  In particular, we  extend results of \cite{blufstein-paris,godelle}.
We recall that $B^+$ is a locally Garside monoid; each
element  $\bb\in B^+$ has a so-called  \emph{Garside normal form} which is
the  unique sequence $(\bb_1,\ldots,\bb_n)$ such
that  $\bb=\bb_1\ldots \bb_n$ where $\bb_i$ is the greatest left-divisor of $\bb_i\ldots
\bb_{n}$ in $B^+$ which is in $\bW$ and $n\ge 0$ is minimal. The following
proposition  summarises some of our results on the retraction of a product
(we postpone the definition of reduced elements and of the
morphism~$\varphi_{\pr(\bb)}$  to  the  next  sections).
\begin{proposition} \label{propIntro1}
Let  $B$ be an Artin  group and $B_I$ be  a standard parabolic subgroup for
some $I\subseteq S$. Let $\bb$ be in $B$.
\begin{enumerate}
\item  (Proposition~\ref{fourre-tout}(\ref{ft2})) If   $\pr(\bb)\in   W_I$   then  
 for $\bb'\in B$ we have $\bpi_I(\bb\bb') =\bpi_I(\bb)\bpi_I(\bb')$.
\item   (Proposition~\ref{fourre-tout}(\ref{ft4}))  Assume  for $J\subseteq S$ that
$\pr(\bb)$  is $I$-reduced-$J$. Then for  $\bj\in B_J$ we have
$\bpi_I(\bb\bj) = \bpi_I(\bb)\varphi_{\pr(\bb)}(\bpi_{J_1}(\bj))$
where $J_1=I^{\pr(\bb)}\cap J$.
\item  (Proposition~\ref{normal form}) For $\bb\in  B^+$, let $(\bb_1,\cdots, \bb_n)$ and
$(\bi_1,\cdots,  \bi_m)$ be the Garside  normal forms of $\bb$
and $\bpi_I(\bb)$ respectively. Then, $m\leq n$ and for $1\leq i\leq m$, the element
$\bpi_I(\bb_1\cdots \bb_i)$ left divides $\bi_1\cdots\bi_i$.
\end{enumerate}
\end{proposition}

In  the next proposition we consider the  case of Artin groups of spherical
type,  that is when  $W$ is finite.  In this case  $W$ has a unique longest
element relative to~$\ell_S$. Its  lift in $\bW$, denoted by
\index{Delta@$\Delta, \Delta_I$}
$\Delta$, is  a  Garside  element  of  $B^+$.  For $I\subseteq S$, the
corresponding  element of $B^+_I$  lifting the longest  element of $W_I$ is
denoted by~$\Delta_I$.

\begin{proposition}  \label{propIntro2}(Proposition~\ref{pi(deltab)}) Assume $W$ is
finite. Then for any $\bb\in B$, $I\subseteq S$ and $i\in\BZ$  we have
$\bpi_I(\bb\Delta^i)=\bpi_I(\bb)\bpi_I(\Delta^i) =\bpi_I(\bb)\Delta_I^i$. 
\end{proposition}

For  $\bb\in B$, define $\bt_I(\bb)$  by $\bb=\bpi_I(\bb)\bt_I(\bb)$. Then, the
map $\bb\mapsto \bt_I(\bb)$ gives a transversal of the right coset $B_I\bb$ in
$B$  (see  Proposition~\ref{Icoset}).  Exchanging  left  and right, one can
define a right retraction $\bpi^r_I(\bb)$ such that $ \bb =
t^r_I(\bb)\bpi^r_I(\bb)$ and $\bb\mapsto t^r_I(\bb)$ gives a transversal of
the  left coset $\bb B_I$ in $B$. These two retractions are closely related
(see Proposition~\ref{rightProj}). We use them to
address  the question  of how  to solve  the double  coset problem in Artin
groups.  A solution to the double coset problem for Coxeter groups is known
since \cite{Bbk}: every double coset has a unique element of minimal length
which  is easy to compute starting  from any (word representing an) element
of  the double coset. Very few results have been obtained for Artin groups.
The  only result on the subject known  to the authors is in the unpublished
article~\cite{Kalkaetall}.  It gives a partial answer  in the case of braid
groups, that is  when $W$  is of  type $A$. Using left and
right  retractions, we obtain a partial result with the same flavour as in the
case of Coxeter groups:

\begin{proposition}\label{propIntro3}
Let  $I,J\subseteq S$ and $\bb\in B$ be such that $\lexp{\pr(\bb)}W_J\cap W_I =\{1\}
$. Then the element $\bb_0=\bt_J^r(\bt_I(\bb))=\bt_I(\bt_J^r(\bb))$ is the
unique element of $B_I\bb B_J$ such that $\bpi_I(\bb_0)=\bpi_J^r(\bb_0)=1$.
\end{proposition}
This is proved after Remark~\ref{intersection empty}
as a consequence of that remark and of the
more general result  Corollary~\ref{unique representative}.

It is a conjecture that the intersection of any two parabolic subgroups is
a parabolic subgroup.  The validity of the conjecture is proved for several families
of Artin groups \cite{AntolinFoniqi,CumplidoetAll3,CumplidoetAll,CumplidoetAll2,Morris-Wright}
but remains open in general.  Using the retraction we address the
conjecture and we show that 
starting from any element $\bb$ in $B$,  the study of $B_I\cap\lexp \bb B_J$ can be
reduced to the case where $\bb = \bp$ is in $P$ and $J = I$
(see Proposition~\ref{not enough}). In this case we prove
\begin{proposition}\label{propIntro4}
Let   $I$  be a  subset   of  $S$  and   $\bp$ be in $P$ such that 
$\pi_I(\bp)=1$; then   $$B_I\cap \lexp\bp B_I=C_{B_I}(\bp)$$
where   $C_{B_I}(\bp) = \{\bb\in B_I\mid \bp\bb = \bb\bp\}$,
the centraliser of~$\bp$  in $B_I$. 
\end{proposition}
 
One  of the main long-standing open problems for braid groups is the existence
of a polynomial-time algorithm for
the conjugacy problem. The partial results obtained on the subject are
related to the Nielsen-Thurston classification (reducible, periodic or 
pseudo-Anosov)  of braids.  The conjugacy  problem, and  even the word problem remain
open in the general case. In the last part of the article, we focus on the
study of the conjugacy class of an element. In the spirit of the
Nielsen-Thurston  classification, we focus  on what might  be called
``\emph{reducible}''
elements.

Denote by \index{lS@$\lS$} $\lS:B\to\BN$ the length function  on  $B$ with respect to the generating
set $\bS^{\pm 1}$. For $\bb\in B$, it is equal to
$\hl(b)$,  the length of any minimal word $b\in(\bS^{\pm1})^*$ representing $\bb$.
By  convexity
any such  word $b$ is written on
the same set $\bI\subseteq\bS$, the image $I\subseteq S$ of which we call the
support of $\bb$, denoted by  $\supp(\bb)$.
The following result is a part of Propositions~\ref{bg=gh},~\ref{minimal  support} and
Corollary~\ref{min in conj}.
\begin{proposition}\label{propIntro5}
Let  $\bb\in  B$ and let $\conj_B(\bb)$ denote the conjugacy class of $\bb$. 
\begin{enumerate}
\item  For $I\subseteq S$, either $\conj_B(\bb)\cap B_I$ is empty or it contains an
element~$\bh$  such that $\lS(\bh)$ is minimal in $\conj_B(\bb)$. 
\item  If  $\lS(\bb)$  is minimal in $\conj(\bb)$ and if $\bb'\in \conj(\bb)$
satisfies $\lS(\bb') = \lS(\bb)$, then $\supp(\bb)$ and $\supp(\bb')$
are conjugate in $W$.
\item If  $\lS(\bb)$  is minimal in $\conj(\bb)$,  then $B_{\supp(\bb)}$ is a minimal
parabolic subgroup containing $\bb$ (among all parabolic subgroups,
not only the standard ones).
\end{enumerate}
\end{proposition}

Note that point (3) above does not state the uniqueness of a minimal parabolic subgroup
since as mentioned above, the question remains open whether the intersection
of two parabolic subgroups is parabolic.

In  Section~\ref{conjectures}, we  will prove  that the  existence of a unique
minimal  parabolic subgroup  containing a  given element  is equivalent to the
following conjecture:
\begin{conjecture}\label{conjectureintr}
Let $\bb\in B$ such that $\lS(\bb)$  is minimal in $\conj(\bb)$
and let $I=\supp(\bb)$, then
any $\bp\in P$ which centralises $\bb$ and is such that $\bpi_I(\bp)=1$
centralises $B_I$.
\end{conjecture}

We  also obtain a partial result on the conjugacy problem: one can say that we
provide a solution to the conjugacy problem for ``\emph{reducible}'' elements in
Artin groups with $S$ finite.

\begin{proposition}(Propositions~\ref{conjugaison V2} and~\ref{determiner aut ruban})\label{propIntro6}
Assume $S$ is finite and let $\Lambda$ be a non-empty set of subsets of $S$ such that for any $I$ in
$\Lambda$,  the word problem and the  conjugacy problem are solvable in $B_I$.
Then the two following problems are solvable.
\begin{enumerate}
\item Given $I$ in $\Lambda$ and a word $b$ on $\bI^{\pm 1}=\bI\cup\bI\inv$, decide whether the
word $b$ represents the unity in $B$.
\item  Given two finite subsets $I, J$ in $\Lambda$ and words $i, j$ on $\bI^{\pm 1}$
and $\bJ^{\pm 1}$, respectively, decide whether or not $i$ and $j$ represent
conjugate elements in $B$.
\end{enumerate}
\end{proposition}
 
The paper is structured as follows.

In  Section~\ref{section1} we  focus on  words.  If $X$ is a
set,  we denote by~$X^*$  the set of  words on $X$.  Using root systems, we
give  an alternate  definition of  the retraction  $\hpi_I$ from $(\bS^{\pm1})^*$,
the words on $\bS\cup\bS\inv$, 
\index{S@$\bS^{\pm1}$,   $(\bS^{\pm1})^*$}
to $(\bS_I^{\pm1})^*$  and extend some results of \cite{godelle}. We obtain new results
for   words  in  $(\bS^{\pm1})^*$  whose  image  in   $W$  is  $I$-reduced-$J$.  

In Section~\ref{section2},  we introduce the retraction $\bpi_I$ on the Artin
group and use the results of the previous section to prove
Propositions~\ref{propIntro1}  and ~\ref{propIntro2}.  Using the  notion of
biclosed  subset  and  results  of  Dyer  we study the compatibility of the
retraction  with lcm's and gcd's. 

In  Section~\ref{section3}  we  turn  to  the  projection~$\bpi^r_I$ and prove
Propositions~\ref{propIntro3}    and~\ref{propIntro4}.

In Section~\ref{conjectures} we consider the conjugacy classes,  prove
Propositions~\ref{propIntro5},~\ref{propIntro6} and state
Conjecture~\ref{conjectureintr} .

Finally, we provide in Section~\ref{section4} (assuming $S$ finite) a new topological version
of  $\bpi_I$ restricted to $\pr\inv(W_I)$, using  the Tits cone and results
of  Van der  Lek \cite{VanderLek}.  

\section{Retraction on words}\label{section1}
For $I\subseteq S$, we  recall  that  an  element  $w\in W$ is said to be
$I$-reduced   (\emph{resp.}  reduced-$I$)   if  $\ell_S(iw)=\ell_S(i)+\ell_S(w)$
(\emph{resp.}  $\ell_S(wi)=\ell_S(i)+\ell_S(w)$)   for   any   $i\in   W_I$,  or
equivalently for any $i\in I$. We say that an element is $I$-reduced-$J$ if
it  is both $I$-reduced  and reduced-$J$. For  $w\in W$  the unique element of
minimum  length in  the coset  $W_I\, w$  is $I$-reduced;  we denote  it by
\index{tI@$t_I$} $t_I(w)$.

\begin{lemma}\label{lemma 1}
If  $w$ is $I$-reduced and $s\in  S$, then the property that $\lexp{w}s\in W_I$
is  equivalent to  $ws$ not  being $I$-reduced.  Furthermore, in  this case
%$\lexp{w}\alpha_s\in \Pi_I$
$\lexp w s\in I$.
\end{lemma}
\begin{proof}Assume first that $\ell_S(ws)>\ell_S(w)$.
If $ws$ is not $I$-reduced, then there exists $s'\in I$ such
that $\ell_S(s'ws)<\ell_S(ws)$, which by the exchange lemma implies that $s'ws$ is
either equal to $w$ or to $\hat ws$ where $\ell_S(\hat w)<\ell_S(w)$ . But $s'ws=\hat
ws$  contradicts the fact that $w$ is  $I$-reduced, so that $s'ws=w$, that
is  $\lexp w  s=s'\in I$.  Conversely if $\lexp w s\in W_I$ then $ws\in W_I w$ so is not
$I$-reduced.

Now if $\ell_S(ws)<\ell_S(w)$, write $w=w's$, then $ws=w'$ is $I$-reduced and by the first
case $\lexp w s=\lexp {w'}s$ is not in $W_I$.
\end{proof}

\index{s@$s_{f(j)}$} \index{piI@$\hat\pi_I$}
\begin{definition}[Retraction]\label{retract}
Write $b\in(\bS^{\pm1})^*$ as $b=\bs_1^{\ve_1}|\ldots|\bs_k^{\ve_k}$ where 
$\ve_i=\pm 1$ and $\bs_i\in\bS$.
For $1\le j\le k$ let $w_j=s_1\ldots s_j$ be the image in $W$ of the $j$ first
letters of $b$.
We apply Lemma~\ref{lemma 1} to $t_I(w_{j-1})$ (with $w_0=1$) to deduce that if
$x=\lexp{t_I(w_{j-1})}s_j\in W_I$ then $x\in I$; we denote then $x$
by $s_{f(j)}$ and say that $j$ is good; otherwise we say that $j$ is bad
and set $s_{f(j)}=1$.
The retraction $\hpi_I(b)$ is the word in $(\bI^{\pm1})^*$ whose letters
are the $\bs_{f(j)}^{\ve_j}$ for $j$ good taken in increasing order.
\end{definition}
This definition is equivalent to that in \cite[Page 1522]{blufstein-paris}
though the notation there makes it a little bit difficult to see.
\begin{remark}\label{obvious}
The following points are clear from the definition.
\begin{enumerate}
\item
$\hpi_I(b)=b$ if and only if $b\in(\bI^{\pm1})^*$.
\item If the
word $b$ is a prefix of the word $b'$ then $\hpi_I(b)$ is a prefix of
$\hpi_I(b')$.
\item $\hpi_I$ sends $\bS^*$ (positive words) to $\bI^*$.
\item $\hpi_I$ commutes with the map $(\bS^{\pm1})^*\to(\bS^{\pm1})^*$
 induced by $\bs\mapsto\bs\inv$ for $\bs\in\bS^{\pm1}$.
\item We denote by \index{l@$\hl$} $\hl(b)$ the length of
$b\in(\bS^{\pm1})^*$.
For any word $b$ we have $\hl(\hpi_I(b))\le \hl(b)$, with equality
if and only if $b\in (\bI^{\pm1})^*$.
\end{enumerate}
\end{remark}

An efficient way to compute $t_I(w_j)$ is given by the following lemma.
\begin{lemma}\label{not good} In the setting of Definition~\ref{retract}, for $1\le j\le k$
we have
\begin{enumerate}
\item $t_I(w_j)$ is equal to the product of the
$s_l$ such that $l$ is bad and $1\le l\le j$,  
\item  $t_I(w_{j-1})s_j=s_{f(j)}t_I(w_j)$.
\item  For all $j$, we have $w_j = s_{f(1)}\cdots s_{f(j)}t_I(w_j)$.
\end{enumerate}
\end{lemma}
\begin{proof}
 We prove items (1) and (2) at the same time by induction on $j$.  Assume that $t_I(w_{j-1})$
is the product of the $s_l$ for $l$ bad and $l\le j-1$. We write $w_{j-1}=v
t_I(w_{j-1})$ with $v\in W_I$. If $j$ is good then $w_j=\lexp{w_{j-1}}s_j
w_{j-1}=\lexp{w_{j-1}}s_j vt_I(w_{j-1})$ with $\lexp{w_{j-1}}s_j v\in W_I$
so that $t_I(w_j)=t_I(w_{j-1})$, hence
$t_I(w_{j-1})s_j=s_{f(j)}t_I(w_{j-1})=s_{f(j)}t_I(w_j)$. If $j$ is not
good then we write $w_j=v t_I(w_{j-1})s_j$. We have to prove that
$t_I(w_{j-1})s_j$ is $I$-reduced, so is equal to $t_I(w_j)$, which is equal to
$s_{f(j)}t_I(w_j)$ since $s_{f(j)}=1$ in this case. 
If $t_I(w_{j-1})s_j$ was not $I$-reduced,
by Lemma~\ref{lemma 1} $\lexp{t_I(w_{j-1})}s_j$ would be in $W_I$ and then
also $\lexp{w_{j-1}}s_j\in W_I$, contradicting $j$ bad.

Now, (3) results immediately by induction from (2).
\end{proof}
\begin{lemma}\label{pi_I(reduced)}
If $b=\bs_1^{\ve_1}|\ldots|\bs_k^{\ve_k}\in(\bS^{\pm1})^*$,
where $\ve_i=\pm 1$, is such that  $s_1|\ldots|s_k$,
is  a reduced decomposition of an $I$-reduced element of
$W$, then $\hpi_I(b)$ is the empty word.
\end{lemma}
\begin{proof}
This follows from the definition of $\hpi_I$ by using that any prefix of an
$I$-reduced element is $I$-reduced and so, using Lemma~\ref{lemma 1},
all indices are bad.
\end{proof}
We  now give another definition of  $\hpi_I$. 

We call reflections of $W$ \index{T@$T, T_I$} the $W$-conjugates of the
elements of $S$ and denote their set by $T$ (and similarly denote by $T_I$ the
set of reflections of $W_I$).

As in \cite[Chapter V, \S 4]{Bbk} we consider a faithful representation of $W$
on  a vector space with a basis $\Pi=\{\alpha_s\}_{s\in S}$ whose elements are
called simple roots. We call root system of $W$ the set \index{Phi@$\Phi,
\Phi^+$} \index{Pi@$\Pi$}$ \Phi$ of images of $\Pi$ under~$W$. We call
positive  roots, the elements of $\Phi$ which are linear combinations of $\Pi$
with positive coefficients, and we denote their set by $\Phi^+$; similarly, we
denote \index{PhiI@$\Phi_I, \Phi_I^+$} \index{PiI@$\Pi_I$} by
$\Phi_I,\Phi_I^+,\Pi_I$ the subsets associated with $W_I$. For any
$\alpha\in\Phi$ we denote the associated reflection by
\index{salpha@$s_\alpha$} $s_\alpha$, and for $t\in T$ we denote the
corresponding positive root by \index{alphat@$\alpha_t$} $\alpha_t$.

Let \index{Phi@$\Phi^*$} $\Phi^*$ denote the set of finite sequences of elements of $\Phi$, and let
\index{Pi@$(\pm\Pi)^*$} $(\pm\Pi)^*$ be  the subset of $\Phi^*$
whose  terms  are  in  $\Pi$ or $-\Pi$.  
For  \index{alpha@$\underline\alpha$} $\underline\alpha=\alpha_1|\ldots|\alpha_n\in\Phi^*$  we  denote by
\index{prod@$\product(\underline\alpha)$} $\product(\underline\alpha)$      the      product      $s_{\alpha_1}\cdots
s_{\alpha_n}\in  W$ and for $w\in  W$ we write $\lexp w{\underline\alpha}$ 
for the sequence
$\lexp  w{\alpha_1}|\ldots|\lexp w{\alpha_n}$.  For $\underline\alpha\in
\Phi^*$  we denote by $\underline\alpha\cap\Phi_I$
the sequence obtained by keeping only the terms lying in $\Phi_I$.
We define \index{N@$\vN$} $\vN:\Phi^*\to \Phi^*$ by $\vN(\alpha_1|\cdots|\alpha_n)=
\alpha_1|\lexp{s_{\alpha_1}}{\alpha_2}|
\cdots|\lexp{s_{\alpha_1}s_{\alpha_2}\ldots s_{\alpha_{n-1}}}{\alpha_n}=
\alpha_1|\lexp{\product(\alpha_1)}\alpha_2|\ldots|
\lexp{\product(\alpha_1|\ldots|\alpha_{n-1})}\alpha_n$. 
\begin{lemma}\label{vecN} $\vN$ has the following properties for any sequence $\underline\alpha$
\begin{enumerate}
\item $\vN(\underline\alpha|\underline\alpha')=\vN(\underline\alpha)|\lexp{\product(\underline\alpha)}\vN(\underline\alpha')$
for any sequence $\underline\alpha'$.
\item $\product(\vN(\underline\alpha))=\product(\underline\alpha)\inv$.
\item For any $w\in W$ we have $\lexp
 w\vN(\underline\alpha)=\vN(\lexp w{\underline\alpha})$.
\item $\vN(\vN(\underline\alpha))=\underline\alpha$.
\end{enumerate}
\end{lemma}
\begin{proof}
Properties  (1) and (3) are clear. It results from
the equality $s_{\lexp w\alpha}=\lexp w s_\alpha$ that
$\product(\vN(\underline\alpha))=s_{\alpha_1}\cdot
\lexp{s_{\alpha_1}}s_{\alpha_2}\cdot\ldots\cdot
\lexp{s_{\alpha_1}s_{\alpha_2}\ldots
s_{\alpha_{n-1}}}s_{\alpha_n}=s_{\alpha_n}s_{\alpha_{n-1}}\ldots
s_{\alpha_1}=
\product(\underline\alpha)\inv$, whence (2). 
Now the $i$-th term of $\vN(\vN(\underline\alpha))$ is the image of
$\lexp{s_{\alpha_1}\cdots s_{\alpha_{i-1}}}\alpha_i$ by the product $\prod_{j=1}^{j=i-1}
\lexp{s_{\alpha_1}\cdots s_{\alpha_{j-1}}}s_{\alpha_i}=s_{\alpha_{i-1}}s_{\alpha_{i-2}}\cdots
s_{\alpha_1}$ hence is equal to $\alpha_i$, whence (4).
\end{proof}
We define \index{p@$p,p^*$} $p:\bS^{\pm1}\to\pm\Pi$ by $\bs\mapsto\alpha_s$ and
$\bs\inv\mapsto-\alpha_s$ for $\bs\in\bS$ and we extend $p$ to a map $p^*$ from
$(\bS^{\pm1})^*$ to $(\pm\Pi)^*$ in the obvious way. 
\begin{lemma}\label{simple}
A positive word $b\in\bS^*$ has a simple image in B (that is an image
in $\bW$) if and only if all terms in $\vN(p^*(b))$ are
positive.
\end{lemma}
\begin{proof}
Let $b=\bs_1|\cdots|\bs_k$ with $\bs_i\in \bS$. As is well known $l(s_1\ldots
s_k)=l(s_1\ldots s_{k-1})+1$ if and only if $\lexp{s_1\ldots s_{k-1}}\alpha_k \in\Phi^+$.
This gives the result by induction on $k$.
\end{proof}
The following proposition is a reformulation of the definition of retraction
in \cite[Definition 2.1]{godelle} and of Definition~\ref{retract}.
\begin{proposition}\label{other way} For $b\in(\bS^{\pm1})^*$,
we can define $\hpi_I(b)$ as
$p^{*-1}(\vN(\vN(p^*(b))\cap \Phi_I))$; this is well-defined since
$\vN(\vN(p^*(b))\cap \Phi_I)\subseteq (\pm \Pi_I)^*$. 
\end{proposition}
\begin{proof}
We retain the notation of Definition~\ref{retract}.
For $i=1,\ldots,k$, let $\alpha_i=\lexp{w_i}p(\bs_i^{\varepsilon_i})$, that is $\alpha_1|\cdots|
\alpha_k=\vN(p^*(b))$. For any good index $i$, let
$A_i= \product((\alpha_1|\cdots|\alpha_{i-1})\cap \Phi_I)$,
so that $\lexp {A_i}\alpha_i$ (for $i$ good) are the
terms of the sequence $\vN(\vN(p^*(b))\cap \Phi_I)$.
We prove by induction on good indices $i$ that for $i$ good we have $A_iw_i=
t_I(w_i)s_i$. Assume the property for $i$ and let $j$ be the next good index.
We have $A_j=A_it_i$, hence $A_jw_j=A_it_iw_is_{i+1}\cdots
s_j=A_iw_{i-1}s_{i+1}\cdots s_j=t_I(w_i)s_{i+1}\cdots s_j$, the second
equality since $t_iw_i=\lexp{w_i}s_i w_i=w_is_i=w_{i-1}$, and the last equality
by induction. We have $A_jw_j=t_I(w_i)
s_{i+1}\cdots s_j=t_I(w_j)s_j$, the second equality by Lemma~\ref{not good},
whence the assertion.

We deduce that for $i$ good we have $\lexp{A_i}\alpha_i=\lexp
{A_iw_i}p(\bs_i^{\varepsilon_i})=     \lexp{t_I(w_i)}p(\bs_i^{\varepsilon_i})$,
which is in $\pm\Pi_I$ by Lemma~\ref{lemma 1}, since ${t_I(w_i)}$ being
$I$-reduced $\lexp{t_I(w_i)} s_i\in I$ is equivalent to
$\lexp{t_I(w_i)}\alpha_{s_i}\in\Pi_I$.
We get the
proposition since by definition the
retraction $\hpi_I(b)$ is $p^{*-1}$ of the sequence  of the
$\lexp{t_I(w_i)}p(\bs_i^{\varepsilon_i})$ for $i$ good.
\end{proof}
\begin{corollary} \label{transitivity} For $I,J$ subsets of $S$ we have
$\hpi_I\circ\hpi_J=\hpi_{I\cap J}$. In particular if $J\subseteq I$ we
have $\hpi_J\circ\hpi_I=\hpi_J$.
\end{corollary}
\begin{proof}
 By Proposition~\ref{other way} we have
 $$\begin{aligned}\hpi_I\circ\hpi_J(b)&=
  p^{*-1}(\vN(\vN(p^*\left (p^{*-1}(\vN(\vN(p^*(b))\cap \Phi_J))\right ))\cap
  \Phi_I))\cr
  &=p^{*-1}(\vN\left (\vN(\vN\left (\vN(p^*(b))\cap \Phi_J\right))\cap \Phi_I\right))\cr
  &=p^{*-1}(\vN(\vN(p^*(b))\cap \Phi_J\cap \Phi_I))\cr
  &=\hpi_{I\cap J}(b)
 \end{aligned}
 $$
 where the equality before last is from Lemma~\ref{vecN}(4).
\end{proof}
We  define  a  map \index{pr@$\hpr$} $\hpr$  from  $(\bS^{\pm1})^*$  to  $W$ by
$\product\circ p^*$ (this factors obviously through the map $\pr:B\to W$ of
the introduction).
\begin{proposition}\label{prtI}
For $b\in (\bS^{\pm1})^*$ we have $\hpr(b)=\hpr(\hpi_I(b))t_I(\hpr(b))$.
\end{proposition}
\begin{proof}
With the notation of Lemma~\ref{not good} we prove by induction on $j$ that  
$\hpr(b_j)=\hpr(\hpi_I(b_j))t_I(\hpr(b_j))$ where
$b_j=\bs_1^{\ve_1}|\cdots|\bs_j^{\ve_j}$, so that $w_j=\hpr(b_j)$. We have
$\hpr(b_j)=\hpr(b_{j-1})s_j=\hpr(\hpi_I(b_{j-1}))t_I(\hpr(b_{j-1}))s_j=
\hpr(\hpi_I(b_{j-1}))s_{f(j)}t_I(\hpr(b_j))$,
the second equality by induction and the third by Lemma~\ref{not good}(2).
We get the result since $\hpr(\hpi_I(b_j))=\hpr(\hpi_I(b_{j-1}))s_{f(j)}$ by
definition of $\hpi_I$.
\end{proof}
\begin{proposition}\label{pi(bb')} We have $\hpi_I(b|b')=\hpi_I(b)|
p^{*-1}(\vN(\vN(\lexp{t_I(\hpr(b))}p^*(b'))\cap \Phi_I))$.
In particular if $\hpr(b)\in W_I$ then $\hpi_I(b|b')=\hpi_I(b)|\hpi_I(b')$.
%In particular if $b$ is an expression for a simple braid and $\hat b$ denotes
%the element of minimal length in the coset $W_I\hpr(b)$, we have
% $\hpi_I(b|b')=\hpi_I(b)|p^{*-1}(\vN(\vN(\lexp{\hat b}p^*(b'))\cap \Phi_I))$.
\end{proposition}
\begin{proof}
 By Lemma~\ref{vecN}(1) we get 
 $\vN(p^*(b|b')\cap \Phi_I)=\vN(p^*(b)\cap \Phi_I)
|\lexp{\hpr(b)}\vN(p^*(b')\cap \Phi_I)$ where we have used that
 $\product(p^*(b))=\hpr(b)$. Using again Lemma~\ref{vecN}(1) we
get $\vN(\vN(p^*(b|b')\cap \Phi_I))=\vN(\vN(p^*(b)\cap \Phi_I)) 
|\lexp{\product(\vN(p^*(b))\cap \Phi_I)}\vN(\lexp{\hpr(b)}\vN(p^*(b'))\cap
\Phi_I)$. Applying $p^{*-1}$to this equality and 
using Proposition~\ref{other way}, we get
$$\hpi_I(b|b')=\hpi_I(b)|p^{*-1}(\lexp{\product(\vN(p^*(b))\cap
 \Phi_I)}\vN(\lexp{\hpr(b)}\vN(p^*(b'))\cap \Phi_I)).$$ 
The right-hand side becomes 
$\hpi_I(b)|p^{*-1}(\lexp{\hpr(\hpi_I(b))\inv}
\vN(\lexp{\hpr(b)}\vN(p^*(b'))\cap \Phi_I))$ using  Lemma~\ref{vecN}(2).
Using Lemma~\ref{vecN}(3)
and that $\hpr(\hpi_I(b))$ normalises $\Phi_I$ this formula becomes
$\hpi_I(b)|p^{*-1}(\vN(\lexp{\hpr(\hpi_I(b))\inv\hpr(b)}\vN(p^*(b'))\cap
 \Phi_I))$.
We get the proposition using once more Lemma~\ref{vecN}(3) and then
Proposition~\ref{prtI}.
\end{proof}
\begin{corollary}\label{pi(b1..bn)} For  $b_1,\ldots b_n\in(\bS^{\pm 1})^*$ 
we have $$\hpi_I(b_1|\cdots|b_n)=\prod_{i=1}^{i=n} p^{*-1}(\vN(\vN(\lexp{t_I(\hpr(b_1\cdots b_{i-1}))}p^*(b_i))\cap\Phi_I)).$$
If all $b_i$ are positive words that have simple images in $B$ the same is true for each term 
of the product on the right-hand side.
\end{corollary}
Note that if all $b_i$ are in $\bS^{\pm1}$ one recovers Definition~\ref{retract}.
\begin{proof}
The formula is obtained by recursively applying Proposition~\ref{pi(bb')}.
We prove the second assertion. By Lemma~\ref{simple} 
a positive word $b$ has a simple image in $B$ if
and only if all terms of $\vN(p^*(b))$ are positive. Since $\vN$ is an
involution by Lemma~\ref{vecN}(4), 
$\vN(\lexp{t_I(\hpr(b_1\cdots b_{i-1}))}p^*(b_i))\cap\Phi_I$ is
equal to $\vN(p^*(b'_i))$, where $b'_i$ is the $i$-th term of the product on the
right-hand side of the statement; 
and all terms of $\vN(\lexp{t_I(\hpr(b_1\cdots b_{i-1}))}p^*(b_i))\cap\Phi_I=
\lexp{t_I(\hpr(b_1\cdots b_{i-1}))}\vN(p^*(b_i))\cap\Phi_I$ are positive by the
simplicity assumption on the image of $b_i$, using the fact that $t_I(\hpr(b_1\cdots b_{i-1}))$
is $I$-reduced.
\end{proof}

\begin{lemma} \label{IredJ} Let $I,J\subseteq S$ and let $w\in W$ be $I$-reduced-$J$;
then $I_1=I\cap\lexp w J$ is the unique maximal subset of $I$ such that
$I_1^w\subseteq J$. Equivalently $J_1= I^w\cap J$ is the unique maximal subset
of $J$ such that $\lexp w J_1\subseteq I$.
\end{lemma}
\begin{proof} This is clear.
\end{proof}

\begin{definition} Let $I\subseteq S$;
an element $w\in W$ which is $I$-reduced and such that $I^w\subseteq S$
is called an $I$-ribbon. If $J=I^w$ we also call it an
$I$-ribbon-$J$ or a ribbon-$J$, in order to specify $J$.

In this situation $w$ is also reduced-$J$ and $w\inv$ induces a bijection
\index{phi@$\varphi_w$}
$\varphi_w:J\xrightarrow\sim I$, which lifts naturally to
a bijection from $(\bJ^{\pm1})^*$ to $(\bI^{\pm1})^*$ that
we still denote by $\varphi_w$.
%\item For $I\subseteq S$ we call simple $I$-ribbon a word $b$ over $\bS\cup\bS\inv$
%such that $p^*(b)$ is a reduced decomposition 
%of an $I$-ribbon (we may say also it is an $I$-ribbon-$J$ when
%$\hpr(b)$ is an $I$-ribbon-$J$).
%
%In this situation $p^*(b)\inv$ induces a bijection
%$\varphi_{\hpr(b)}:J\xrightarrow\sim I$, which can be lifted to
%a bijection from words on $\bJ\cup\bJ\inv$ to words on $\bI\cup\bI\inv$ that
%we still denote by $\varphi_{\hpr(b)}$.
%\item 
%We say that a word $b$ is an $I$-ribbon if $b=b_1|\ldots|b_k$ where there
%is  a sequence $I=I_0,I_1,\ldots,I_k$  of subsets of $S$ such that $b_i$ is a
%simple  $I_{i-1}$-ribbon  whose  image  $w_i\in  W$  satisfies
%$I_{i-1}^{w_i}=I_i$.
%
%We also say that $b$ is an $I$-ribbon-$I_k$ if we want to specify the endpoint
%image of $I$.
%In this situation we let $\varphi_b=\varphi_{b_1}\circ\ldots\varphi_{b_k}$.
\end{definition}
\begin{remark}
Under the assumptions of Lemma~\ref{IredJ}, the element $w$ is a
 $I_1$-ribbon-$J_1$;
indeed we have $I_1^w=J_1$ and $w$ is $I_1$-reduced-$J_1$.
Moreover $I_1$ and $J_1$
are the maximal subsets of $I$ and $J$ respectively with that property.
\end{remark}
\begin{proposition}\label{ribbon}
For $I,J\subseteq S$, let $b,b'\in(\bS^{\pm1})^*$ be such that $\hpr(b)$ is an
$I$-ribbon-$J$.
Then  $\hpi_I(b|b')=\hpi_I(b)|\varphi_{\hpr(b)}(\hpi_J(b'))$.
\end{proposition}
\begin{proof}
Proposition~\ref{pi(bb')} gives
$\hpi_I(b|b')=\hpi_I(b)|p^{*-1}(\vN(\vN(\lexp{\hpr(b)}p^*(b'))\cap \Phi_I))$.
Using that $\Phi_I^{\hpr(b)}=\Phi_J$ the second term of the right-hand side is equal
to
$$p^{*-1}(\lexp{\hpr(b)}\vN(\vN(p^*(b'))\cap \Phi_J))=
 \varphi_{\hpr(b)}(p^{*-1}(\vN(\vN(p^*(b'))\cap \Phi_J)))=\varphi_{\hpr(b)}(\hpi_J(b')).$$
\end{proof}
If we lift an expression $w=s_1\cdots s_k$ to a word $\hat w\in \bS^*$, the
image \index{N@$N$}  $N(w)$  in   $\BZ/2\BZ[T]$  of   $\vN(p^*(\hat  w))$,  obtained  by
associating to each root the corresponding reflection, depends only on $w$.
Indeed $N(w)=\sum_{i=1,\ldots,k}\lexp{s_1\cdots s_{i-1}}s_i \in\BZ/2\BZ[T]$
can  also be considered as  the subset of $T$  consisting of elements which
appear   an  odd  number  of   times  in  the  sequence  $\{\lexp{s_1\cdots
s_{i-1}}s_i\}_{1\le i\le k}$; it is known that $N$ depends only on $w$ (see
\cite[Chapter  IV, \S 1.4,  Lemme 2]{Bbk}). It  is also shown  in loc.\ cit.\ 
that  $N$ is injective  and  that  $|N(w)|=\ell_S(w)$.  For  $v,v'\in W$ one has
$N(vv')=N(v)+\lexp v N(v')$.

The  next proposition is \cite[Lemma 2]{solomon}. The proof we give, in the
setting  of  the  elementary  Coxeter  group  theory,  is  shorter that the
original proof.
\begin{proposition}\label{solomon} Let $I,J$ be subsets of $S$ and
$w\in W$ be $I$-reduced-$J$. Then $W_J\cap W_I^w=W_{J_1}$ where $J_1=I^w\cap J$.
\end{proposition}
\begin{proof}
If $y\in W_J\cap W_I^w$ then $wy\in W_Iw$ so we can write $wy=xw$ with $x\in
W_I$. We get thus $N(w)\coprod\lexp wN(y)=N(x)\coprod\lexp xN(w)$ where
the sums are disjoint because $w$ is reduced on both sides. But since
$w$ is $I$-reduced $N(w)$ does not meet $W_I$, thus $N(w)\subseteq \lexp xN(w)$
and we must have equality since these sets have same cardinality. Thus in
turn we must have $\lexp wN(y)=N(x)$. We claim that by induction on $\ell_S(y)$
the assertion $\lexp wN(y)\subseteq W_I$ implies $y\in W_{J_1}$. Let
$y=sy'$ where $s\in J$ and $\ell_S(y')=\ell_S(y)-1$. By Lemma~\ref{lemma 1} 
the fact that $\lexp ws\in W_I$ implies $\lexp ws\in I$ thus $s\in J_1$,
and $ws=tw$ for some $t\in I$.
Using $N(sy')=s\coprod \lexp sN(y')$ we get $\lexp{ws}N(y')\subseteq W_I$,
and since $ws=tw$ this implies $\lexp wN(y')\subseteq W_I$ 
and we conclude by induction.
\end{proof}
\begin{corollary}\label{cor_solomon}
Let $I,J$ be subsets of $S$ and
$w\in W$ be $I$-reduced-$J$. Then $\Phi_J\cap \Phi_I^w=\Phi_{J_1}$ where $J_1=I^w\cap J$.
\end{corollary}
\begin{proof}
We have $\alpha\in\Phi_J\cap\Phi_I^w$ if and only if $s_\alpha\in W_J\cap W_I^w=W_{J_1}$ by
Proposition~\ref{solomon}, which in turn is equivalent to $\alpha\in \Phi_{J_1}$.  
\end{proof}
\begin{proposition} \label{pi_I(rb)V2} Let $I,J\subseteq S$ and
$b\in(\bS^{\pm1})^*$ be such that $\hpr(b)$ is  $I$-reduced-$J$.
Let $J_1=I^{\hpr(b)}\cap J$;
Then for any word $b'\in(\bJ^{\pm1})^*$,  we have $\hpi_I(b|b') = 
\hpi_I(b)|\varphi_{\hpr(b)}(\hpi_{J_1}(b'))$.    
\end{proposition} 
\begin{proof}
Using that $t_i(\hpr(b)=\hpr(b)$ since $\hpr(b)$ is $I$-reduced, and using
Lemma~\ref{vecN}(3) twice, Proposition~\ref{pi(bb')} becomes in our case
$\hpi_I(b|b')=\hpi_I(b)|p^{*-1}(\lexp{\hpr(b)}\vN(\vN(p^*(b'))\cap
\Phi_I^{\hpr(b)}))$. Now by assumption $\vN(p^*(b'))\subseteq \Phi_J$, and by
Corollary~\ref{cor_solomon} we have $\Phi_J\cap\Phi_I^{\hpr(b)}=\Phi_{J_1}$.
The second term is thus 
$p^{*-1}(\lexp{\hpr(b)}\vN(\vN(p^*(b'))\cap\Phi_{J_1}))$,
and   we   conclude   by
Proposition~\ref{pi(bb')} since
$p^{*-1}(\lexp{\hpr(b)}\vN(\vN(p^*(b'))\cap\Phi_{J_1}))$   is  equal  to  $
\varphi_{\hpr(b)}(p^{*-1}(\vN(\vN(p^*(b'))\cap\Phi_{J_1})))$,  that  is  to
$\varphi_{\hpr(b)}(\hpi_{J_1}(b'))$.
\end{proof}

\section{Retraction on braids}\label{section2}
We say that a subgroup of $W$ is parabolic if it is conjugate to a standard  
parabolic subgroup of $W$.

As a direct application of Proposition~\ref{solomon},  we get
\begin{lemma}\label{solomon2} 
the  intersection of any two parabolic
subgroups  $D$ and $Q$ of $W$ is a parabolic subgroup of each of them.
\end{lemma}
\begin{proof} Up to conjugacy, we may assume that $Q$ is standard, say $Q =
W_J$.  Assume $D = W_I^v$. Write $w =  iwj$ with $i\in W_I$, $j\in W_J$ and
$w$ $I$-reduced-$J$. Then $D\cap Q = (W_J\cap W_I^w)^v$, and we conclude by
Proposition~\ref{solomon}.
\end{proof}
\begin{lemma}\label{intersection parabolics}
For any parabolic subgroups $Q,D$ of $W$ where $D$ is dihedral
the set $T\cap Q\cap D$ is either empty, or is reduced to one
reflection or is the set $T\cap D$ of all reflections of $D$.
\end{lemma}
\begin{proof}
By Lemma~\ref{solomon2}  the  intersection of the two parabolic
subgroups  $D$ and $Q$ of $W$ is a parabolic subgroup of each of them. The
possible  parabolic subgroups of  $D$ are $\{1\}$,  a subgroup generated by
one reflection, or $D$ itself. This gives the three cases of the statement.
\end{proof}
\begin{lemma} \label{braid relation} Let $b,b'\in(\bS^{\pm1})^*$ and assume
one of the following for two words $r$ and $r'$ in $(\bS^{\pm1})^*$:
\begin{enumerate}
\item $r,r'$ are the two members of a braid relation,
\item $r=\bs|\bs\inv$ or
$\bs\inv|\bs$ and $r'$ is the empty word,
\end{enumerate}
then the two words 
$\hpi_I(b|r|b')$ and $\hpi_I(b|r'|b')$ are either equal or differ by a 
relation  in $(\bI^{\pm1})^*$ of the same type as the relation $r\equiv r'$.
\end{lemma}
\begin{proof}
Applying Corollary~\ref{pi(b1..bn)} to $b|r|b'$ and $b|r'|b'$, 
since $\hpr(r)=\hpr(r')$, it is sufficient to compare
$p^{*-1}(\vN(\vN(\lexp{t_I(\hpr(b))}p^*(r))\cap \Phi_I))$ and
$p^{*-1}(\vN(\vN(\lexp{t_I(\hpr(b))}p^*(r'))\cap \Phi_I))$.
Using Lemma~\ref{vecN}(3) it is equivalent to compare for any $I$-reduced 
element $w\in W$ the words
$p^{*-1}(\vN(\lexp w\vN(p^*(r))\cap\Phi_I))$ and 
$p^{*-1}(\vN(\lexp w\vN(p^*(r'))\cap\Phi_I))$.

In case (1)
the  sequences in $\vN(p^*(r))$  and $\vN(p^*(r'))$ are  reversed from each
other  and  consist  of  all  the  positive  roots  of  a standard dihedral
parabolic.  Hence  the  roots  in  $\vN(p^*(r))\cap\lexp{w\inv}\Phi_I$  and
$\vN(p^*(r'))\cap\lexp{w\inv}\Phi_I$  are  all  the  positive  roots in the
intersection  of that  dihedral subgroup  with $\lexp{w\inv}W_I$.  By 
Lemma~\ref{intersection  parabolics} there are three cases for this set of roots:
either it is empty or it is reduced to one root or it is the set of all the
positive roots of a dihedral parabolic subroup of $\lexp{w\inv}W_I$.

In  the  first  two  cases  we  get  clearly  the equality of the sequences
$\vN(\lexp w\vN(p^*(r))\cap\Phi_I)$ and $\vN(\lexp
w\vN(p^*(r'))\cap\Phi_I)$.   In   the   third   case   we  have  $\vN(\lexp
w\vN(p^*(r))\cap\Phi_I)=\vN(\lexp   w\vN(p^*(r)))=\lexp   w   p^*(r)$   and
$\vN(\lexp    w\vN(p^*(r'))\cap\Phi_I)=\vN(\lexp   w\vN(p^*(r')))=\lexp   w
p^*(r')$  and  these  two  sequences  are  reversed from each other and are
sequences  in $\Pi_I^*$. Thus applying $p^{*-1}$ gives the two members of a
braid relation in $B_I$, whence the result.

In case (2) we have $\lexp w\vN(p^*(r))=\alpha|-\alpha$ for some root
$\alpha$. If $\alpha\notin \Phi_I$, then $\lexp w\vN(p^*(r))\cap\Phi_I)$
is the empty sequence and 
$\hpi_I(b|r|b')$ and $\hpi_I(b|r'|b')$ are equal. If $\alpha\in \Phi_I$ then 
the two words
$p^{*-1}(\vN(\lexp w\vN(p^*(r))\cap\Phi_I))$ and 
$p^{*-1}(\vN(\lexp w\vN(p^*(r'))\cap\Phi_I))$
differ by a relation of type $\bs|\bs\inv\equiv \emptyset$ or
$\bs\inv|\bs\equiv\emptyset$ with $\bs\in\bI$.
\end{proof}
\begin{proposition}\label{bien defini} For $b\in(\bS^{\pm1})^*$ of
image $\bb$ in $B$, the image in $B$ of $\hpi_I(b)$ depends only on $\bb$.
 We call it \index{piI@$\bpi_I$} $\bpi_I(\bb)$.
\end{proposition}
\begin{proof}
This is an immediate consequence of Lemma~\ref{braid
relation}.
\end{proof}

The following proposition was first proved by Van der Lek.  In \cite{VanderLek}, $S$ is assumed to be finite but the
result for infinite $S$ is an almost immediate consequence. 
\begin{proposition}\label{vdl}For  $I\subseteq S$, the natural  morphism from the braid group
of $W_I$ to $B_I$ is an isomorphism.
\end{proposition}
\begin{proof}The morphism is clearly surjective.
We show the injectivity. Let $b$ be a word in $(\bI^{\pm1})^*$. We have to show that if 
$b\equiv\emptyset$ in $B$, then the same equivalence occurs using only
relations in $(\bI^{\pm1})^*$. Let $b=b_1\equiv b_2\equiv\ldots\equiv\emptyset$
be a chain of equivalences in $B$ by relations of the type considered in
Lemma~\ref{braid relation}. Apply $\hpi_I$ to this chain. The first and
the last term are unchanged, and by Lemma~\ref{braid relation} each other
equivalence is mapped to an equivalence in $(\bI^{\pm1})^*$. This proves the
result.
\end{proof}
Statements of Section~\ref{section1} translate to the following (sometimes
weaker) properties of $\bpi_I$:
\begin{proposition}\label{fourre-tout} For $\bb\in B$ we have the following:
\begin{enumerate}
\item\label{ft6}
$\bpi_I(\bb)=\bb$ if and only if $\bb\in B_I$.
\item\label{ft7} $\bpi_I$ sends $B^+$  to $B_I^+$.
\item\label{ft8} If $\bb\in B^+$ left-divides $\bb'\in B^+$
then $\bpi_I(\bb)$ left-divides $\bpi_I(\bb')$.
\item\label{ft9} $\bpi_I$ commutes with the automorphism $\inverse: B\to B$ \index{inv@$\inverse$}
 induced by $\bs\mapsto\bs\inv$ for $\bs\in\bS$.
\item\label{ft5} If $\bb\in\bW$ and $\pr(\bb)$ is $I$-reduced, then $\bpi_I(\bb)=1$.
\item\label{ft1} $\pr(\bb)=\pr(\bpi_I(\bb))t_I(\pr(\bb))$.
\item\label{ft2} Assume that $\pr(\bb)\in W_I$; then for $\bb'$ in $B$  we have
$\bpi_I(\bb\bb')=\bpi_I(\bb)\bpi_I(\bb')$. In particular the restriction of
$\bpi_I$ to $\pr\inv(W_I)$ is a group homomorphism.
\item\label{ft3} Let $\bb'\in B$ and let  $I,J\subseteq S$ be such that $\pr(\bb)$ is an
$I$-ribbon-$J$. Then $\bpi_I(\bb\bb')=
\bpi_I(\bb)\varphi_{\pr(\bb)}(\bpi_J(\bb'))$.
\item\label{ft4}  For $I,J\subseteq S$ such that $\pr(\bb)$ is $I$-reduced-$J$ and any
$\bj\in B_J$, we have $\bpi_I(\bb\bj) =
\bpi_I(\bb)\varphi_{\pr(\bb)}(\bpi_{J_1}(\bj))$ where
$J_1=I^{\pr(\bb)}\cap J$.
\item\label{ft10}  If $\bb$ is a product of $n$ elements of $\bW$, then
$\bpi_I(\bb)$ is a product of at most $n$ elements of $\bW_I$.
\end{enumerate}
\end{proposition}
\begin{proof} (\ref{ft6}),  and (\ref{ft9}) are (1) and (4) of 
Remark~\ref{obvious}. (\ref{ft7}), (\ref{ft8}) follow respectively from 
(3) and (2) of Remark~\ref{obvious}. (\ref{ft5}) is Lemma~\ref{pi_I(reduced)},
(\ref{ft1}) is Proposition~\ref{prtI}, (\ref{ft2}) is Proposition~\ref{pi(bb')}, (\ref{ft3}) is
Proposition~\ref{ribbon}, and (\ref{ft4}) is Proposition~\ref{pi_I(rb)V2}.
Finally (\ref{ft10}) follows from Corollary~\ref{pi(b1..bn)}.
\end{proof}
\begin{lemma}\label{mult2}   Let  $\bb\in  B$  be  such  that  $\pr(\bb)\in
N_W(W_I)$;  then there exists a unique $w\in W$ such that $w$  is
an $I$-ribbon-$I$ and such that
$\pr(\bb)\in W_I w$.  For any $\bb'\in B$ we have
$\bpi_I(\bb\bb')=\bpi_I(\bb)\varphi_w(\bpi_I(\bb'))$.
\end{lemma}
\begin{proof} The unicity of $w$ comes from the fact that $w$ is the
$I-$reduced element in $W_I\pr(\bb)$.
By \cite[Lemme 6.1.7]{digne-michel} (the result is implicit in \cite{Brink-Howlett})
the $I$-reduced element $w$ in $W_I\pr(\bb)$ is an $I$-ribbon-$I$. Let $\bw\in \bW$ be
the lift of $w$ and $\bb_1=\bb\bw\inv$; then $\pr(\bb_1)\in W_I$ and, 
by items (\ref{ft2}) and (\ref{ft3}) of Proposition~\ref{fourre-tout}, we have
$\bpi_I(\bb\bb')=\bpi_I(\bb_1)\bpi_I(\bw\bb')=\bpi_I(\bb_1) \varphi_w(\bpi_I(\bb'))$.
The equality of the lemma
follows since  $\bpi_I(\bb)=\bpi_I(\bb_1)$ by the last equality applied with
$\bb'=1$.
\end{proof}
\begin{definition}\label{tailInArtinGroup} Given $I\subseteq S$,
\begin{itemize}
\item \index{tI@$\bt_I$}
For $\bb$ in $B$, we define $\bt_I(\bb)=\bpi_I(\bb)\inv\bb$.
\item
For $w\in W$, we define \index{piI@$\pi_I$} $\pi_I(w)=w t_I(w)\inv$.
\end{itemize}
\end{definition}
The following lemma shows that the definitions $t_I,\bt_i$ and $\pi_I,\bpi_I$
are compatible.
\begin{lemma} \label{Icoset}
Let $I\subseteq S$ and $\bb$ be in $B$. 
\begin{enumerate}
\item  $\bpi_I(\bt_I(\bb))=1$.
\item $\pr(\bt_I(\bb))=t_I(\pr(\bb))$. In particular, $\pr(\bt_I(\bb))$  is $I$-reduced.
\item $\pr(\bpi_I(\bb))=\pi_I(\pr(\bb))$.
\item If $\bb'$ is in $B$, then $B_I\bb = B_I\bb'\iff \bt_I(\bb) = \bt_I(\bb')$.
\end{enumerate}
\end{lemma} 
Note that (3) can be  seen as a  part of the commutative diagram in Proposition~\ref{commutative1}.  

\begin{proof}   We   have   $\bpi_I(\bb)   =  \bpi_I(\bpi_I(\bb)\bt_I(\bb))
=\bpi_I(\bpi_I(\bb))\bpi_I(\bt_I(\bb))   =  \bpi_I(\bb)\bpi_I(\bt_I(\bb))$,
 the  second equality by Proposition~\ref{fourre-tout}(\ref{ft2}).  So $\bpi_I(\bt_I(\bb)) = 1$. The
second and third items come from
$\pr(\bpi_I(\bb))\pr(\bt_I(\bb))=\pr(\bb)=\pr(\bpi_I(\bb))t_I(\pr(\bb))$
where  the first  equality is  by definition  and the second by 
Proposition~\ref{fourre-tout}(\ref{ft1}). For (4), clearly, if $\bt_I(\bb)=\bt_I(\bb')$
then $B_I\bb =
B_I\bb'$. Conversely, assume $\bi \bb = \bi'\bb'$ with $\bi,\bi'$ in $B_I$.
Then  $\bi\bpi_I(\bb)\bt_I(\bb)=\bi'\bpi_I(\bb')\bt_I(\bb')$, thus  to prove
that $\bt_I(\bb)=\bt_I(\bb')$ it suffices to prove that
 $\bi\bpi_I(\bb)=\bi'\bpi_I(\bb')$.   But   by   Proposition~\ref{fourre-tout}(\ref{ft2})  we  have
$\bi\bpi_I(\bb)=\bpi_I(\bi\bb)$  and $\bi'\bpi_I(\bb')=\bpi_I(\bi'\bb')$ so
they are equal by assumption.
\end{proof}
Note that $\bt_I(B^+)\not\subset B^+$. For instance, in type $A_2$ with
$\bS=\{\bs,\bt\}$ and $I=\{s\}$, we have $\pi_I(\bt\bt\bs)=\bs$, thus
$\bt_I(\bt\bt\bs)=\bs\inv\bt\bt\bs$.

\begin{proposition}\label{commute pr}
For $w\in W$, $\pi_I(w)$ is the unique
element $v\in W_I$ such that $N(w)\cap T_I=N(v)$.
\end{proposition}
\begin{proof}Since $N$ is injective it
is  sufficient  to  prove  that  $v$  defined  by $w=vt_I(w)$ satisfies
$N(v)=N(w)\cap  T_I$. We have  $N(w)=N(v) + \lexp  v N(t_I(w))$. We want to
show  that $\lexp  v N(t_I(w))$  does not  meet $W_I$  or equivalently that
$N(t_I(w))$ does not meet $W_I$. But if $t_I(w)=s_1\ldots s_k$ is a reduced
expression  this is  equivalent to  $s_1\ldots s_j  s_{j-1}\ldots s_1\notin
W_I$  for  all  $j$  or  equivalently  $s_1\ldots  s_j\notin  W_I s_1\ldots
s_{j-1}$, which is a consequence of $s_1\ldots s_j$ being $I$-reduced.
\end{proof}

Since $B^+$ is a locally Garside monoid, elements have a greatest common
left-divisor  (left-gcd) and if they have  a common right-multiple they have a
right-lcm.  Each element has a Garside  normal form, which for  $\bb\in B^+$ is a
sequence $(\bb_1,\ldots,\bb_n)$ with $\bb_i\in\bW$ such that
$\bb=\bb_1\cdots\bb_n$,  uniquely defined by the  property that $\bb_i$ is the
greatest left-divisor in $\bW$ of $\bb_i \bb_{i+1}$ and by the
property that no $\bb_i$ is $1$.
\begin{proposition} \label{normal form}
For $\bb\in  B^+$, let $(\bb_1,\cdots, \bb_n)$ and
$(\bi_1,\cdots,  \bi_m)$ be the Garside  normal forms of $\bb$
and $\bpi_I(\bb)$ respectively. Then, $m\leq n$ and for $1\leq i\leq m$, the element
$\bpi_I(\bb_1\cdots \bb_i)$ left divides $\bi_1\cdots\bi_i$.
\end{proposition}
\begin{proof}
 As  a particular case of Proposition~\ref{fourre-tout}(\ref{ft10})
the normal form of $\bpi_I(\bb_1\ldots \bb_n)$ has $n$ terms or less, since
the product of $n$ simple braids has not more than $n$ terms in its normal
 form. Now by Proposition~\ref{fourre-tout}(\ref{ft8}) we have that $\bpi_I(\bb_1\cdots\bb_i)$
left-divides $\bi_1\cdots\bi_m$ and since it has $i$ simple terms by
Proposition~\ref{fourre-tout}(\ref{ft10}), by
\cite[III, 1.14]{DDGKM} it divides $\bi_1\cdots\bi_i$.
\end{proof}

The  next two propositions assume $W$ finite.  In this case $B^+$ is a Garside
monoid with Garside element the lift $\Delta\in\bW$ of the longest element of
$W$.  The elements of $\bW$ are  both the left-divisors and the right-divisors
of $\Delta$. Similarly $B_I^+$ is a Garside monoid with Garside element $\Delta_I$, the lift
of the longest element of $W_I$.
\begin{proposition} \label{pi(deltab)} Assume that $W$ is finite. 
Then for any $\bb\in B$, $I\subseteq S$ and $i\in\BZ$ we have
$\bpi_I(\bb\Delta^i)=\bpi_I(\bb)\Delta_I^i$.
\end{proposition}
\begin{proof}
One deduces the proposition for $i<0$ from the result for $-i$ by writing 
$\bpi_I(\bb)=\bpi_I(\bb\Delta^i\Delta^{-i})=\bpi_I(\bb\Delta^i)\Delta_I^{-i}$.
The proposition for $i\in\BN$
is obtained by iterating the case where $i=1$, so we assume now $i=1$.

We  denote by $\varphi$  the automorphism of  $B$ induced by conjugating by
$\Delta$.  Define $\bw$  by $\Delta=\Delta_I\bw$;  then $w=\pr(\bw)$  is a
 $I$-ribbon-$\varphi(I)$.
We  get $\bpi_I(\bb\Delta)=\bpi_I(\Delta\varphi\inv(\bb))=
\bpi_I(\Delta_I\bw\varphi\inv(\bb))=
\Delta_I\bpi_I(\bw\varphi\inv(\bb))=
\Delta_I\varphi_w(\bpi_{\varphi(I)}(\varphi\inv(\bb)))$ where the third
equality is by Proposition~\ref{fourre-tout}(\ref{ft2}) and the last
by Proposition~\ref{fourre-tout} items (\ref{ft5}) and (\ref{ft3}).  
Now the automorphism
$\varphi$ clearly commutes with $\bpi_I$, that is for any $\bb\in B$ we have
$\varphi_w\bpi_{\varphi(I)}(\varphi\inv(\bb))=\varphi_I(\bpi_I(\bb))$  where
$\varphi_I$ is automorphism of $B_I$ induced by $\Delta_I$. The proposition
follows since $\Delta_I\varphi_I(\bpi_I(\bb))= \bpi_I(\bb)\Delta_I$.
\end{proof}
\begin{proposition} Assume that $W$ is finite. We recall that for
$\bb\in B$ we define $\sup(\bb)=\min\{j\mid \bb\inv\Delta^j\in B^+\}$ and 
$\inf(\bb)=\max\{j\mid \Delta^{-j}\bb\in B^+\}$. We
have $\sup(\bpi_I(\bb))\le\sup(\bb)$ and $\inf(\bpi_I(\bb))\ge\inf(\bb)$.
\end{proposition}
\begin{proof}
The reversed of the word $\bb\inv$ is $\inverse(\bb)$, where $\inverse$ is as
in Proposition~\ref{fourre-tout}(\ref{ft9}), and $\Delta$ is equal to its reversed, thus
the condition $\bb\inv\Delta^j\in B^+$ is equivalent
to $\Delta^j\inverse(\bb)\in B^+$, which is itself equivalent to
$\inverse(\bb)\Delta^j\in B^+$.
We have $\bpi_I( \inverse(\bb)\Delta^j)=\bpi_I(\inverse(\bb))\Delta_I^j=
\inverse(\bpi_I(\bb))\Delta_I^j\in B_I^+$,
the first equality by Proposition~\ref{fourre-tout}(\ref{ft2}) and the second
one by Remark~\ref{fourre-tout}(\ref{ft9}); 
arguing as at the beginning but in $B_I$  the last is equivalent
to $\bpi_I(\bb)\inv\Delta_I^j\in B_I^+$, which concludes.

The proof for $\inf$ is similar but simpler: we don't have to use reversing.
\end{proof}
\begin{lemma}\label{piw=1}
Any $\bb\in B$ can be written $\bp\bi\bw$ where $\bp$ is pure,
$\bi\in\bW_I$ (the lift of $W_I$) and $\pr(\bw)$ is $I$-reduced; in such a
decomposition we have $\bpi_I(\bb)=\bpi_I(\bp)\bi$.
\end{lemma}
\begin{proof}
It is clear that any element $\bb\in B$ can be written $\bp\bi\bw$ as above.
 By Proposition~\ref{fourre-tout}(\ref{ft2}) we have $\bpi_I(\bb)=\bpi_I(\bp)\bpi_I(\bi)\bpi_I(\bw)=
 \bpi_I(\bp)\bi\bpi_I(\bw)$  and  by Lemma~\ref{fourre-tout}(\ref{ft5}) $\bpi_I(\bw)=1$.
\end{proof}

The following lemma is \cite[Proposition 2.3(3)]{blufstein-paris}.
\begin{lemma}\label{pi_i(p)}
If $\bp$ is a pure braid then $\bpi_I(\bp)$ is pure.
\end{lemma}
\begin{proof}
We have $\bp=\bpi_I(\bp)\bt_I(\bp)$ so it is sufficient to show that $\bt_I(\bp)$
is pure. By Lemma~\ref{Icoset}(2) we have $\pr(\bt_I(\bp))=t_I(\pr(\bp))=1$,
whence the result.
\end{proof}

The following lemma is \cite[Lemma 3.2]{godelle}
\begin{lemma}\label{lemme 9} Let $I, J$ be subsets of $S$ and let
$\bb\in B$ be such that $\pr(\bb)$ is $I$-reduced-$J$; then $\bpi_I(\bb B_J)=
 \bpi_I(\bb)B_{I_1}$ where $I_1=\lexp {\pr(\bb)} J\cap I$.
\end{lemma}
\begin{proof}
Applying  Proposition~\ref{fourre-tout}(\ref{ft4})
we get $\bpi_I(\bb B_J)=\bpi_I(\bb)\varphi_{\pr(\bb)}(\bpi_{J_1}(B_J))=
\bpi_I(\bb)\varphi_{\pr(\bb)}(B_{J_1})=\bpi_I(\bb)B_{I_1}$ where
$J_1=I_1^{\pr(\bb)}$.
\end{proof}
The following is \cite[Theorem 1.1]{blufstein-paris}.
\begin{proposition}\label{parabolic inside} Let $I,J$ be subsets of $S$ and
let  $\bb\in B$ be such  that $\lexp \bb  B_J\subseteq B_I$. Then there exists
$\bb'\in  B_I$ and $K\subseteq I$ such that $\lexp \bb B_J=\lexp{\bb'}B_K$.
In other terms a parabolic subgroup of $B$ lying inside $B_I$ is a parabolic subgroup of $B_I$.
\end{proposition}
\begin{proof}
Let $w=\pr(\bb)$. Applying $\pr$ we get $\lexp wW_J\subseteq W_I$. Write $w=iw' j$ where $i\in W_I, j\in W_J$ and $w'$ is $I$-reduced-$J$. We still have
$\lexp{w'}W_J\subseteq   W_I$  and  by Lemma~\ref{lemma 1} we  have
$\lexp{w'}J=K$  for some $K\subseteq I$. Lifting $i,j,w'$ to $\bi,\bj,\bw'\in
\bW$  we can write  $\bb=\bi\bp\bw'\bj$ where $\bp$  is pure. We still have
$\lexp{\bp\bw'}B_J\subseteq   B_I$,   and   since   $\lexp{w'}J=K$   we  have
$\lexp{\bw'}B_J=B_K$. We thus get $\lexp\bp B_K\subseteq B_I$. Thus $\lexp\bp
B_K=\bpi_I(\lexp\bp  B_K)=\bpi_I(\bp)B_K\bpi_I(\bp)\inv$,
the last equality from Lemma 3 since $\bp$ is pure and
$\pr(B_K)\subseteq W_I$.
This  proves the statement with $\bb'=\bi\bpi_I(\bp)$.
\end{proof}
\subsection*{The map $\bN$}
The following proposition is \cite[Proposition~2.1]{digne}. We recall it here,  as it is unpublished.
\begin{proposition}\label{digne2.1}
\phantom{a}
\begin{enumerate}
\item There exists a well defined map \index{N@$\bN$} $\bN:B\to \BZ[T]$ such that
$$\bN(\bs_1^{\varepsilon_1}\cdots \bs_k^{\varepsilon_k})=\sum_{i=1}^{i=k} \varepsilon_i
\lexp{s_1\cdots s_{i-1}}s_i$$
for $\bs_i\in\bS$ and $\varepsilon_i=\pm1$.
\item The map $\bN\rtimes\pr$ is group homomorphism from $B$ to
$\BZ[T]\rtimes W$, where the action of $W$ on $\BZ[T]$
extends linearly the conjugation action of $W$ on $T$.
\end{enumerate}
\end{proposition}
\begin{proof}
We prove that the map $\bs\mapsto (s,s)$ for $\bs\in\bS$ extends to 
a group morphism $f:B\to \BZ[T]\rtimes W$.
Composing this morphism with the first projection will give a map $\bN$ which satisfies
(1). It is sufficient to show that the braid relations ($**$) are satisfied, that is for any
$\bs,\bt\in\bS$, we have the equality $f(\bs)f(\bt)\cdots=f(\bt)f(\bs)\cdots$, where there
are $m_{s,t}$ terms on both sides. In $\BZ[T]\rtimes W$ the first component of
$(s,s)(t,t)\cdots$ (with $m_{s,t}$ factors) is the sum of all reflections of the dihedral
group generated by $s$ and $t$ and the second component is the longest element of this
dihedral group. So we get the same result by swapping $s$ and $t$, whence the
proposition.
\end{proof}
The map $\bN$ lifts $N$ to the braid group:
for $\bb\in B$ the set $N(\pr(\bb))$
is the reduction modulo 2 of $\bN(\bb)$.
For $\bw\in\bW$ we can
identify $\bN(\bw)$ to $N(w)$:
one gets $\bN(\bw)$ from $N(w)$ by lifting $\BZ/2\BZ$ to $\{0,1\}\subseteq\BZ$.
If we interpret $\bN$ as a map $B\to
\BZ[\Phi^+]$, then $\bN(\bb)$ is the sum of the sequence $\vN(p^*(b))$ for any word
$b$ representing $\bb$.
\begin{lemma}\label{generators}We say that $\bw\in \bW$ is $I$-reduced if $\pr(\bw)$ is $I$-reduced.
\begin{itemize}
\item The pure braid group $P$ is generated by $P_I=P\cap B_I$ and
the elements $\bi\bw\bs^2\bw\inv\bi\inv$  for $\bs\in\bS$ and
$\bi\in \bW_I$ and $\bw\in\bW$ such that $\bw\bs$ is an $I$-reduced element.
\item The group $\pr\inv(W_I)$ is
generated by $B_I$ and the elements $\bw\bs^2\bw\inv$ with $\bs\in\bS$ and 
$\bw\bs\in\bW$ an $I$-reduced element.
\item If  $\bs\in\bS$ and if $\bw\bs\in\bW$ is an $I$-reduced element then
$\bpi_I(\bw\bs^2\bw\inv)=1$.
\end{itemize}
\end{lemma}
\begin{proof}
We first prove that $P$ is generated by the $\bw\bs\bw\inv$ with
$\bw,\bs,\bw\bs\in\bW$: if
in the Reidemeister-Shreier method we take $\bW$ as
representatives of the $P$-cosets in~$B$, we get that $P$ is generated by the
$\bw\bs\bv\inv$ where $\bw\in \bW$ and $\bs\in\bS$ and $\bv\in\bW$ is the representative of $\bw\bs$; if
$\bw\bs$ is in $\bW$ we have $\bv=\bw$, otherwise $\bw=\bu\bs$ for some
$\bu\in\bW$ and $\bv=\bu$ so that $\bw\bs\bv\inv=\bu\bs\bu\inv$. We get then
the first item of the lemma by writing
$\bw=\bi\bv$ where $\bi\in\bW_I$ and $\bv$ is $I$-reduced.

We get the second item from the first
using that an element of $\pr\inv(W_I)$ is the product of an element
of $B_I$ by a pure braid.

For the third item,
by Proposition~\ref{fourre-tout}(\ref{ft2})  we have 
$\bpi_I(\bw\bs^2\bw\inv)\bpi_I(\bw\bs\inv)=\bpi_I(\bw\bs)$, and  by 
Proposition~\ref{fourre-tout}(\ref{ft5}) we obtain $\bpi_I(\bw\bs\inv)=\bpi_I(\bw\bs)=1$. Thus, $\bpi_I(\bw\bs^2\bw\inv)=1$.
\end{proof}
\begin{lemma}\label{bsb}
For $\bb\in B$ and $\bs\in \bS$, we have $\bN(\bb\bs^2\bb\inv)=2t$ where
$t=\pr(\bb\bs\bb\inv)$.
\end{lemma}
\begin{proof}
Using that $\bN\rtimes \pr$ is a morphism, that is 
$$(\bN\rtimes\pr)(\bb\bb')=(\bN(\bb)+\lexp{\pr(\bb)}\bN(\bb'),\pr(\bb\bb')),$$
and setting $b=\pr(\bb)$, we get
\begin{equation}\label{equation}
(\bN(\bb\bs^2\bb\inv),1)=(\bN(\bb),b)(\bN(\bs^2),1)(\bN(\bb\inv),b\inv)
\end{equation}
 From $(0,1)=(\bN(\bb),b)(\bN(\bb\inv),b\inv)$ we have
 $\lexp{b}{\bN}(\bb\inv)=-\bN(\bb)$. Using this and $\bN(\bs^2)=2s$
in equation (\ref{equation}) we get 
 $$(\bN(\bb\bs^2\bb\inv),1)=(
\bN(\bb)+2\cdot\lexp{b}s+\lexp{b}{\bN}(\bb\inv),1)
=(2\cdot\lexp{b}s,1)=(2t,1).$$
\end{proof}
\begin{proposition}\label{commutative1}
The following diagram is commutative:
$$\begin{CD}
B@>\bN\times\pr>>\BZ[T]\rtimes W\\
 @VV\bpi_I V @VV {\mathrm{proj}}_{T_I}\times \pi_I V\\
B_I @>\bN\times \pr>> \BZ [T_I]\rtimes W_I
\end{CD}$$
\end{proposition}
\begin{proof}
Writing  an element  of $\bb\in  B$ as  $\bb_1\bb_2\bb_3$ where  $\bb_1$ is
pure,  $\bb_2\in B_I$ and $\bb_3\in\bW$  is $I$-reduced,
 and applying Proposition~\ref{fourre-tout}(\ref{ft2}), we can reduce to the case of an element
of one of the above 3 forms.

For an element of $B_I$ the proposition is trivial since $\bpi_I$ is the
identity on $B_I$.

For $\bw\in\bW$  an $I$-reduced element, we know the 
commutativity with $\pr$ by Lemma~\ref{Icoset}(3),
and we have $\bpi_I(\bw)=1$. All that remains is to prove that $\bN(\bw)$ has trivial
coefficients on $T_I$. Since $\bw\in\bW$,  the lift of $N(w)$ is $\bN(\bw)$
and we conclude since $N(w)\cap W_I=\emptyset$ by Proposition~\ref{commute
pr}.

For  a  pure  braid,  we  can,  using
Proposition~\ref{fourre-tout}(\ref{ft2}) and Lemma~\ref{generators}, reduce to an element of the form $\bw\bs^2\bw\inv$
where $\bw\bs$ is $I$-reduced.
By Lemma~\ref{bsb} we have $\bN(\bw\bs^2\bw\inv)=2t$ where
 $t=\pr(\bw\bs\bw\inv)$.
By Lemma~\ref{generators} $\bpi_I(\bw\bs^2\bw\inv)=1$
and by Lemma~\ref{lemma 1} we have $t\notin W_I$.
Thus the commutation holds.
\end{proof}

\subsection*{Closed subsets}
For $w\in W$ the set $N(w)$ is in bijection
with the set $\Phi_w=\{\alpha\in\Phi^+\mid
w\inv(\alpha)\notin\Phi^+\}$ since $N(w)=\{s_\alpha\mid\alpha\in\Phi_w\}$. 
We now recall results of \cite{dyer}.
We say that a set $\Gamma\subseteq\Phi^+$ is closed if for any
$\alpha,\beta\in\Gamma$ and $a,b\in\BR_{>0}$, if $a\alpha+b\beta\in\Phi$
then $a\alpha+b\beta\in\Gamma$. We call closure of $\Gamma$, denoted by
\index{Ga@$\overline\Gamma$} $\overline\Gamma$, the smallest closed set containing $\Gamma$. We
say that $\Gamma$ is biclosed if it is closed and its complement $\complement\Gamma$ in
$\Phi^+$ is closed. Thanks to the bijection between reflections and positive roots we
can use the same words for subsets $A\subseteq T$
(closed, notation $\bar A$ for the closure, 
biclosed and notation \index{C@$\complement$} $\complement A$ for the complement in $ T$).
The notion of closed for a set $A\subseteq T$ can be given in purely
group-theoretic terms: a set $A$
is closed if and only if given any two reflections $s,s'\in W$, the intersection of $A$
with the dihedral group $\langle s,s'\rangle$ is closed.  In a dihedral group
$W=\langle s,s'\rangle$, the group theoretic definition of closed is as follows: 
we give a total order on $ T$ by $s < ss's <ss'ss's<\ldots< s'ss'<s'$
(exchanging $s$ and $s'$ gives the opposite
order); then a set $A$ is closed if and only if whenever it contains $t,t'\in T$ with
$t<t'$ it contains any $t''$ such that $t<t''<t'$ (see
\cite[(2.2)]{dyer1} for the equivalence of the group-theoretic condition with
closed assuming the existence of a reflection order and see
\cite[Proposition 2.3]{dyer1} for the existence of a reflection order).
\begin{lemma} \label{inclusionN} For $\bw,\bw'\in\bW$ it is equivalent that
$\bw$ left-divides $\bw'$ in the monoid $B^+$ or that $N(w)\subseteq N(w')$.
\end{lemma}
\begin{proof}
The assumption that $\bw$ left-divides $\bw'$ is equivalent in $W$ 
to the existence of $v\in W$ such that $w'=wv$ and
$\ell_S(w)+\ell_S(v)=\ell_S(w')$. By the formula $N(wv)=N(w)\cup \lexp wN(v)$ it is clear
that left-divisibility implies $N(w)\subseteq N(w')$. We show the converse
by induction on $\ell_S(w')$. Let $\bs\in \bS$ be a left-divisor of $\bw$, equivalently
$\ell_S(sw)=\ell_S(w)-1$ or equivalently $s\in N(w)$ 
(see \cite[Lemma 2.1.6(ii)]{digne-michel}). It follows that
$\bs$ is also a left-divisor of $\bw'$. Let $w=sw_1$ and $w'=sw'_1$.
 From the formulae $N(w)=\{s\}\cup\lexp sN(w_1)$ and 
$N(w')=\{s\}\cup\lexp sN(w'_1)$ it follows that $N(w_1)\subseteq N(w'_1)$
and we conclude by induction.
\end{proof}
More generally, using that divisibilty in $\bW$ projects to the
weak order in $W$, we have
\begin{proposition}[Dyer]\label{Dyer}\phantom{}
\begin{itemize}
\item A set $A\subseteq T$ is biclosed finite if and only if it is an $N(w)$ for
some $w\in W$.
\item Let $\bw,\bw'\in\bW$ be two simple braids which have a common right-multiple
in the monoid $B^+$. 
Then their right-lcm is an element of $\bW$ such that 
$$N(\pr(\rightlcm(\bw,\bw')))= \overline{N(w)\cup N(w')}.$$
\item Let $\bw,\bw'\in\bW$ be two simple braids. Then 
$$\complement N(\pr(\leftgcd(\bw,\bw')))=\overline{\complement(N(w)\cap N(w'))}.$$
\end{itemize}
\end{proposition}
\begin{proof}
See \cite[4.1, 1.5]{dyer}.
\end{proof}
\begin{proposition} Let $\bw,\bw'\in\bW$ be two simple braids. 
Then for $I\subseteq S$ we have
$\bpi_I(\rightlcm(\bw,\bw'))=\rightlcm(\bpi_I(\bw),\bpi_I(\bw'))$ and
$\bpi_I(\leftgcd(\bw,\bw'))=\leftgcd(\bpi_I(\bw),\bpi_I(\bw'))$.
\end{proposition}
\begin{proof}
By   Proposition~\ref{commute   pr},  $\bpi_I(\bw)$   is  characterised  by
 $N(\pr(\bpi_I(\bw)))=N(w)\cap T_I$,  that    is    $(\Phi_I)_{\bpi_I(w)}=
\Phi_w\cap\Phi_I$. 

Thus, for the $\rightlcm$, from Proposition~\ref{Dyer} we have to prove that
taking the closure  of a  set of  positive roots  commutes with  intersecting with
$\Phi_I^+$,  which  is  clear  since  if  for  $\alpha,\beta\in\Phi^+$  and
$a,b\in\BR_{>0}$,  the  root  $a\alpha+b\beta$  is  in  $\BR_{\ge 0}\Pi_I$ then
$\alpha$ and $\beta$ also. This can be written as a formula
$\overline{A\cap T_I}=\overline A\cap T_I$.

For the $\leftgcd$, if we set $A=N(w)\cap N(w')$, we have
$\overline{\complement A}\cap T_I=\overline{\complement A\cap T_I}=
\overline{\complement_I(A\cap T_I)}$ where the first equality is the previous
observation and in the second equality $\complement_I$ denotes the complement in $ T_I$;
this proves the proposition.
\end{proof}
By Proposition~\ref{fourre-tout}(\ref{ft8}), for any positive braids $\bb$ and $\bb'$, 
$\rightlcm(\bpi_I(\bb),\bpi_I(\bb'))$ is a left-divisor of
$\bpi_I(\rightlcm(\bb,\bb'))$ and 
$\bpi_I(\leftgcd(\bb,\bb'))$  is a left-divisor of
$\leftgcd(\bpi_I(\bb),\bpi_I(\bb'))$;
but equality does not generally hold for non-simple braids.

For example, in the braid group of $\mathfrak S_3$
with $\bS=\{\bs,\bt\}$, taking $I=\{s\}$, $\bb=\bs\bt\bt$ and
$\bb'=\bt$ we have $\bpi_I(\bb)=\bs$, $\bpi_I(\bb')=1$, hence  
$\rightlcm(\bpi_I(\bb),\bpi_I(\bb'))=\bs$ but
$\bpi_I(\rightlcm(\bb,\bb'))=\bpi_I(\bs\bt\bs\bt\bs)=\bs\bs$.

A counterexample for the gcd is obtained with $\bb=\bt\bt\bs$, $\bb'=\bs$ and $I=\{s\}$.
We have $\bpi_I(\bb)=\bpi_I(\bb')=\bs$ hence $\leftgcd(\bpi_I(\bb),\bpi_I(\bb'))=\bs$
but $\leftgcd(\bb,\bb')=1$ whose retraction is $1$. 
\section{Retraction on the other side}\label{section3}
Definition~\ref{retract} of $\bpi_I$ is based on right-cosets $W_Iw\in
W_I\backslash W$ and the ``left $I$-tail'' $t_I(w)$. It
is well-behaved regarding the left-divisibility of words: if $b\in(\bS^{\pm 1})^*$ 
is a prefix of $b'$
then $\hpi_I(b)$ is a prefix of $\hpi_I(b')$.  Using left-cosets
$wW_I\in W/W_I$ and the counterpart \index{trI@$t^r_I$} $t^r_I(w)$  of $t_I(w)$,
we obtain another retraction that we
denote by \index{pirI@$\bpi^r_I$} $\bpi^r_I$.  A natural
question is the connection between the two retractions. The next lemma gives an answer to this question.
\begin{lemma}\label{rightProj}
We have $\bpi^r_I(\bb)=\rev(\bpi_I(\rev(\bb))) =
(\bpi_I(\bb\inv))\inv$, where \index{rev@$\rev$} $\rev$ is the unique
antiautomorphism of $B$ that fixes $\bS$. 
\end{lemma}
\begin{proof}
The first equality is clear and the second equality
 follows from Proposition~\ref{fourre-tout}(\ref{ft9}) and from the equality
 $\bb\inv=\inverse(\rev(\bb))$.
\end{proof}

For $\bb\in B$ we define 
$\bt_I^r(\bb)  = \bb\bpi^r_I(\bb)\inv$; it satisfies the counterpart result of
Lemma~\ref{Icoset}.

The right-retraction counterpart of Proposition~\ref{fourre-tout}(\ref{ft4}) is
\begin{proposition}\label{counterpart}
Let $I,J\subseteq S$ and $\bb\in B$ be such that $\pr(\bb)$ is $I$-reduced-$J$;
then for $\bi\in B_I$ and for $I_1=I\cap \lexp{\pr(\bb)}J$, we have
$\bpi_J^r(\bi\bb)=\varphi\inv_{\pr(\bb)}(\bpi^r_{I_1}(\bi))\bpi_J^r(\bb)$.
\end{proposition}

\begin{proposition}\label{ribboninv} For $I,J\subseteq S$ and $\bb\in B$
be such that $\pr(\bb)$ is an $I$-ribbon-$J$, we have 
$\bpi_I(\bb)=\varphi_{\pr(\bb)}(\bpi_J(\bb\inv))\inv$, in particular
$$\bpi_I(\bb) = 1 \iff \bpi_J(\bb^{-1}) = 1.$$
\end{proposition}
\begin{proof}
We have $1 = \bpi_I(1) =
\bpi_I(\bb\bb^{-1})=\bpi_I(\bb)\varphi_{\pr(\bb)}(\bpi_J(\bb^{-1}))$, 
where the last equality is by Proposition~\ref{fourre-tout}(\ref{ft3}), 
whence the proposition.
\end{proof}
 
%\begin{lemma}\label{J'_1  dans  J_1}  For  $I,J\subseteq  S$ and $w\in W$ an
%$I$-reduced-$J$ element, let $J_1=I^w\cap J$. If $vw$ is an $I'_1$-ribbon-$J'_1$
%for some $v\in W_I$, $I'_1\subseteq I$ and $J'_1\subseteq J$ then $J'_1\subset
%J_1$.
%\end{lemma}
%\begin{proof}
%Any  element of $J'_1$  can be written  $s^{vw}$ for some  $s\in I'_1$. By
%Lemma~\ref{lemma 1}, $s^v$ lies in $I$, thus $s^{vw}$ lies in $I^w\cap J=J_1$.
%\end{proof}
\begin{proposition}\label{ajout F} Let $I,J \subseteq S$  and $\bb$ in $B$ be such
that $\pr(\bb)$ is  reduced-$J$. Let $J_1=I^{\pr(\bb)}\cap J$; then 
$\pr(\bt_I(\bb))$ is a ribbon-$J_1$ and for any $\bj\in B_J$ 
one has $$\bpi_I(\bb\bj)=\bpi_I(\bb)\varphi_{\pr(\bt_I(\bb))}(\bpi_{J_1}(\bj))=
\bpi_I(\bb)\bpi_{I_1}(\bt_I(\bb)\bj)$$ where
$I_1=\lexp{\pr(\bt_I(\bb)}J_1=I^{\pr(\bpi_I(\bb))}\cap \lexp{\pr(\bt_I(\bb))}J$.
In particular, when $\bpi_I(\bb) = 1$ we have $\bpi_I(\bb\bj) = \bpi_{I_1}(\bb
 \bj)$.    
\end{proposition}
\begin{proof}
For any $s\in J_1$ we have
$\lexp{\pr(\bb)}s\in I$. Writing $\pr(\bb)=
\pr(\bpi_I(\bb))\pr(\bt_I(\bb))$ we deduce that $\lexp{\pr(\bt_I(\bb))}s$
lies in $W_I$, hence in $I$ by Lemma~\ref{lemma 1}.
Since $\bt_I(\bb)$ is $I$-reduced-$J$ we deduce that it 
is an $I_1$-ribbon-$J_1$.
Writing $\bpi_I(\bb\bj)=\bpi_I(\bb)\bpi_I(\bt_I(\bb)\bj)$,
 we can apply Proposition~\ref{fourre-tout}(\ref{ft4}) with $\bt_I(\bb)$ for $\bb$,
whence $\bpi_I(\bb\bj)=
\bpi_I(\bb)\varphi_{\pr(\bt_I(\bb))}(\bpi_{J_1}(\bj))$. The equality 
$\varphi_{\pr(\bt_I(\bb))}(\bpi_{J_1}(\bj))=\bpi_{I_1}(\bt_I(\bb)\bj)$ comes from
 Proposition~\ref{fourre-tout}(\ref{ft4}) with $\bt_I(\bb)$ for $\bb$ and $I_1$ for $I$, taking in account
 the fact that $\bpi_{I_1}(\bt_I(\bb))=1$ by Corollary~\ref{transitivity}.
\end{proof}
The right-counterpart of Proposition~\ref{ajout F} is
\begin{proposition} \label{pi_I(rb)V2modif} Let $I,J\subseteq S$ and $\bb$ in
$B$ be such
that $\pr(\bb)$ is  $I$-reduced.   Let $I_1=I\cap \lexp{\pr(\bb)}J$; then
for any $\bi\in B_I$ we have $\pi^r_J(\bi\bb) =   \pi^r_{J_1}\big(\bi
t_J^r(\bb)\big) \pi^r_J(\bb)$, where
$J_1=I_1^{\pr(\bt_J^r(\bb)}=I^{\pr(\bt^r_I(\bb))}\cap \lexp{\pr(\bpi^r_I(\bb))}J$.
In particular, when $\pi^r_J(\bb) = 1$ we have
$\pi^r_J(\bi\bb) = \pi^r_{J_1}(\bi \bb)$. 
\end{proposition}

\begin{proposition}\label{prepdbcoset}Let $I,J\subseteq S$ and $\bb\in B$;
set $\bb_1 = \bt^r_J(\bt_I(\bb))$.  Then 
\begin{enumerate}
\item $\bpi_I(\bb_1)=\bpi^r_J(\bb_1)=1$. In particular $\pr(\bb_1)$ is $I$-reduced-$J$.
\item  Let $I_1= I\cap\lexp{\pr(\bb_1)}J$ and $J_1=I^{\pr(\bb_1)}\cap J$;
then 
$$\bpi^r_J(\bb)=\varphi\inv_{\pr(\bb_1)}(\bpi^r_{I_1}(\bpi_I(\bb)))\bpi^r_J(\bt_I(\bb)).$$
\end{enumerate}
\end{proposition}
\begin{proof}
 First,  $\bpi_I(\bt_I(\bb)) = 1$ and
$\pr(\bt_I(\bb))$ is $I$-reduced by Lemma~\ref{Icoset}. For a similar reason,
$\pr(\bb_1)$ is reduced-$J$, hence it is  $I$-reduced-$J$ since it
has to left-divide (be smaller for the weak order than) 
$\pr(\bt_I(\bb))$ in $W$.  By definition of $\bb_1$ we have
$\bpi^r_J(\bb_1)=1$.
Since $J_1\subseteq J$ one has
$\bpi^r_{J_1}(\bb_1) = \bpi^r_{J_1}(\bpi^r_J(\bb_1))= 1$. Applying 
Lemma~\ref{rightProj} we get
 $\bpi_{J_1}(\bb_1^{-1}) = 1$ and by Propositions~\ref{ribboninv} and~\ref{ajout F}, we get that
 $\bpi_{I_1}(\bb_1) = 1$. Since $\pr(\bb_1)$ is $I$-reduced-$J$, applying
 again Lemma~\ref{Icoset}, we get $\pr(\bt_I(\bb_1)) = \pr(\bb_1)$ and
 $\bpi_I(\bt_I(\bb_1)) = 1$. By Proposition~\ref{ajout F} we deduce
 that $\bpi_I\big(\bt_I(\bb_1)\bpi_J^r(\bt_I(\bb))\big) =
 \bpi_{I_1}\big(\bt_I(\bb_1)\bpi_J^r(\bt_I(\bb))\big)$, and therefore belongs to
 $B_{I_1}$.  But now $1 = \bpi_I(\bt_I(\bb)) =
\bpi_I\big(\bpi_I(\bb_1)\bt_I(\bb_1)\bpi_J^r(\bt_I(\bb))\big) = \bpi_I(\bb_1)
\bpi_I\big(\bt_I(\bb_1)\bpi_J^r(\bt_I(\bb))\big)$. So $\bpi_I(\bb_1)$ is equal to
$\big(\bpi_I\big(\bt_I(\bb_1)\bpi_J^r(\bt_I(\bb))\big)\big)^{-1}$ and belongs to $B_{I_1}$ too. This imposes
$\bpi_I(\bb_1)  = \bpi_{I_1}(\bpi_I(\bb_1))=\bpi_{I_1}(\bb_1)=1$, the second
equality by Corollary~\ref{transitivity}.
This concludes the proof of (1).

To prove (2) we decompose $\bb$ as $\bpi_I(\bb)\bb_1\bpi_J^r(\bt_I(\bb))$.
 Applying $\bpi^r_J$, we get
$\bpi_J^r(\bb)=\bpi_J^r(\bpi_I(\bb)\bb_1)\bpi_J^r(\bt_I(\bb))$  (since a  term
 in $B_J$  factors  out  on  the  right  in  $\bpi_J^r$)  and  then we apply
Proposition~\ref{counterpart}  with  $\bpi_I(\bb)$  for $\bb'$ and
$\bb_1$ for $\bb$.
\end{proof}

\begin{corollary} Let $I,J\subseteq S$ and $\bb\in B$ be such that
$\pr(\bb)$ is a $I$-ribbon-$J$; then
$\bpi_I(\bb)=\varphi_{\pr(\bb)}(\bpi^r_J(\bb))$.
\end{corollary}
\begin{proof}
 By Proposition~\ref{ribboninv}, we have
$\bpi_I(\bb)=\varphi_{\pr(\bb)}(\bpi_J(\bb\inv))\inv=\varphi_{\pr(\bb)}(\bpi_J^r(\bb)\inv)\inv$,
the last equality by Proposition~\ref{rightProj}. 
\end{proof}

In general $\bt^r_J(\bt_I(\bb))\ne \bt_I(\bt^r_J(\bb))$. For instance in the 
braid group of
type $A_2$,  with $\bS=\{\bs,\bt\}$, let $I=\{s\}= J$  and
$\bb=\bs\bt^2\bs$. We have $\bt^r_J(\bt_I(\bb)) = \bs\inv\bt^2\bs$  and
$\bt_I(\bt^r_J(\bb)) = \bs\bt^2\bs\inv$  which differ since $\bs^2\bt^2\neq
\bt^2\bs^2$. In particular the double-coset $B_I\bb B_J$ generally does not contain
a unique element  $\bb_0$ such that $\bpi_I(\bb_0) = \bpi_J(\bb_0) = 1$.

\begin{proposition}\label{propdbcoset}
Let  $I,J\subseteq S$. Let $\bb_0, \bb_1$ be in $B$ such that $\bpi_I(\bb_0) =
\bpi^r_J(\bb_0)  = \bpi_I(\bb_1) = \bpi^r_J(\bb_1) = 1$. Let $I_1=I\cap\lexp{\pr(\bb_0)}J$ and
$J_1=I^{\pr(\bb_0)}\cap J$; then the following are equivalent:
\begin{enumerate}
\item $B_I\bb_0 B_J = B_I\bb_1 B_J$.
\item $B_{I_1}\bb_0 B_{J_1} = B_{I_1}\bb_1 B_{J_1}$.
\end{enumerate}
Furthermore for any $\bi\in B_I$ and  $\bj\in B_J$ such that 
$\bi\bb_0  = \bb_1\bj$ we have $\bi\in B_{I_1}$, $\bj\in B_{J_1}$
and $\bi=\varphi_{\pr(\bb_0)}(\bj)$.  
\end{proposition}
\begin{proof} It is clear that (2)$\Rightarrow$(1).
Assume $B_I\bb_0 B_J = B_I\bb_1 B_J$. Let $\bi$ in $B_I$ and
$\bj$ in $B_J$ be such  that  $\bi\bb_0  = \bb_1\bj$.  We
have $\bj=\bpi^r_J(\bb_1\bj)=\bpi^r_J(\bi\bb_0)=
\varphi\inv_{\pr(\bb_0)}(\bpi^r_{I_1}(\bi))$, the last equality by the
counterpart of Proposition~\ref{fourre-tout}(\ref{ft4}). 
In  particular, $\bj$ lies in $B_{J_1}$. Symmetrically
$\bi=\pi_I(\bi\bb_0)=\pi_I(\bb_1\bj)=\varphi_{\pr(\bb_1)}(\bpi_{I^{\pr(\bb_1)}\cap J}(\bj))\in
B_{I\cap \lexp{\pr(\bb_1)} J}$, the last equality by Proposition~\ref{fourre-tout}(\ref{ft4}). 
Since $\pr(\bb_0)$ and $\pr(\bb_1)$ are $I$-reduced-$J$ elements
in the same double coset of $W_I\backslash W/W_J$ they are equal. Thus
$I\cap \lexp {\pr(\bb_1)}J=I_1$ and $\bi$ is in $B_{I_1}$.
\end{proof}

\begin{corollary}\label{unique representative}
Let  $I,J\subseteq S$ and $\bb_0\in B$ be such that
$\bpi_I(\bb_0) = \bpi^r_J(\bb_0) = 1$.  Then the
following are equivalent
\begin{enumerate}
\item \label{cor_asser1} $\pr(\bb_0 \bJ\, \bb_0\inv \cap  \bI) = \lexp{\pr(\bb_0)} J \cap  I$.
\item \label{cor_asser2} for every  $\bb$ in $B_I\bb_0 B_J$, if $\bpi_I(\bb) =
 \bpi^r_J(\bb) = 1$ then $\bb = \bb_0$.
\item \label{cor_asser3} for every $\bb$ in $B_I\bb_0 B_J$, one has
$\bt_I(\bt^r_J(\bb))  = \bt^r_J(\bt_I(\bb)) = \bb_0$.
\end{enumerate}
\end{corollary}

\begin{proof} Set $w  = \pr(\bb_0)$,  $J_1 = I^w\cap J$ and $I_1 = \lexp w J \cap
I$.  Assume  (\ref{cor_asser1})  holds.  Let  $\bb$  in  $B$ be such that
$\bpi_I(\bb)  = \bpi^r_J(\bb)  = 1$  with $B_I\bb_0  B_J = B_I\bb
B_J$. By proposition~\ref{propdbcoset}(2) we
get $\bi\bb_0 = \bb\bj$ for some
$\bi \in B_{I_1}$ and  $\bj\in B_{J_1}$ such that $\bi = \varphi_{w}(\bj)$. But
by~(\ref{cor_asser1}),  $\bb_0$ conjugates $\bJ_1$ to $\bI_1$ thus 
the  conjugation  by  $\bb_0$  in  $B$  induces the
one-to-one  map  $\varphi_w$  from $B_{J_1}$  to $B_{I_1}$. So 
$\bb_0\inv\bi\bb_0=\bj$ thus $\bb_0=\bb$.
Thus    (\ref{cor_asser2})    holds.    Now,   (\ref{cor_asser2})   implies
(\ref{cor_asser3}). Indeed, by Proposition~\ref{prepdbcoset}(1)
 the element $\bt^r_J(\bt_I(\bb))$ verifies the assumption for $\bb$ in (2), 
and   by   symmetry  the same holds for $\bt_I(\bt^r_J(\bb))$.
Finally assume (\ref{cor_asser3}). Clearly
$\pr(\bb_0 \bJ\,  \bb_0\inv \cap  \bI) \subseteq  w J w\inv \cap I$
always  holds. Let $s$ in $I_1$ and $s'$ in $J_1$ be such that $ws'w\inv =
s$.  Set $\bb_1 = \bs\inv \bb_0\bs'$.  We have
$\bpi_I(\bb_1) =\bs\inv\bpi_I(\bb_0\bs')$ and by Proposition~\ref{pi(bb')}
we have $\bpi_I(\bb_0\bs')=\bs$ since $t_I(\pr(\bb_0))=w$.
Thus $\bpi_I(\bb_1)=1$. Similarly
$\bpi_J^r(\bb_1)  = 1$ and therefore $\bb_1 = \bt_I(\bt^r_J(\bb_1)) = \bb_0$, the
last equality by (3).
Hence,  $\bb_0 = \bs\inv \bb_0\bs'$. and  $s$ belongs to $\pr(\bb_0 \bJ\,
\bb_0\inv \cap \bI)$.
\end{proof}

\begin{remark}\label{intersection empty} Since one has always  
$\pr(\bb_0 \bJ\, \bb_0\inv \cap  \bI) \subseteq \lexp {\pr(\bb_0)} J \cap  I$,
the three items of Corollary~\ref{unique representative} are true in the particular case
$\lexp {\pr(\bb_0)} J \cap  I = \emptyset$.
\end{remark}

\begin{proof}[Proof of Proposition~\ref{propIntro3}]
If $\lexp {\pr(\bb)} W_J \cap  W_I = \{1\}$,  then $\lexp {\pr(\bb_0)} W_J \cap  W_I = \{1\}$. As $\pr(\bb_0)$ is $I$-reduced-$J$,  by Proposition~\ref{solomon} the last equality is equivalent to the equality  $\lexp {\pr(\bb_0)} J \cap  I = \emptyset$, and we conclude using Corollary~\ref{unique representative} and  the above remark~\ref{intersection empty}.
\end{proof}

\begin{lemma}\label{ribbon in B}
Let   $I,J$  be  subsets   of  $S$  and  let $\bw\in\bW$ be the lift of an $I$-ribbon-$J$;
Then $\lexp\bw B_J= B_I$.
\end{lemma}
\begin{proof}
For $j\in J$ we have $wj=iw$ for some $i\in I$. Since $w$ is $I$-reduced-$J$, the lift to
$\bW$ of
$wj$ is $\bw\bj$ and the lift of $iw$ is $\bi\bw$, thus $\bw\bj=\bi\bw$ and
$\lexp\bw\bj$ is in $\bI$. 
\end{proof}
The following proposition shows various reductions of the problem of
the intersection of two parabolic subgroups $\lexp{\bb}B_J$ and $B_I$. (1) shows that we can
assume $\bpi_I(\bb)=\bpi_J^r(\bb)=1$, (2) shows that we can also assume that
$\pr(\bb)$ is an $I$-ribbon-$J$, (3) shows that we can also assume that $\bb$ is
pure and $I=J$ (see Proposition~\ref{propIntro4}) and (4) shows that if $I$
is finite, then we can assume that $B_J$ is a minimal parabolic subgroup containing
$B_J\cap\lexp{\bb}B_J$.

\begin{proposition}\label{not enough}
Let   $I,J$  be  subsets   of  $S$  and   $\bb\in  B$. 
\begin{enumerate}
\item  Let  $\bb_0=\bt_J^r(\bt_I(\bb))$.  Then
$\pi_I(\bb_0)=\pi_J^r(\bb_0)=1$, and $\lexp\bb B_J\cap
B_I=\lexp{\pi_I(\bb)}(\lexp{\bb_0}B_J\cap B_I)$
\item
Assume $\pi_I(\bb)=\pi_J^r(\bb)=1$. Then 
$\lexp\bb B_J\cap B_I=\lexp\bb B_{J_1}\cap B_{I_1}$, where
$I_1= I\cap\lexp{\pr(\bb)}J$ and $J_1= I^{\pr(\bb)}\cap J$ (so that
$\pr(\bb)$ is an $I_1$-ribbon-$J_1$).
\item Assume that $\pi_I(\bb)=1$ and that $\pr(\bb)$ is an $I$-ribbon-$J$. Then
if $\bw\in \bW$ lifts $\pr(\bb)$  and $\bp$ is the (pure) element such that $\bb=\bp\bw$,
we have $\pi_I(\bp)=1$ and $B_I\cap\lexp \bb B_J=B_I\cap \lexp\bp B_I=C_{B_I}(\bp)$.
\item For $I$ and $\bp$ as in item (3); assume $I$ finite, then the intersection
$B_I\cap \lexp \bp B_I$ is conjugate
in $B_I$ to $B_J\cap\lexp {\bp'} B_J$ for some $J\subseteq I$ and $\bp'\in P$
such that
$\bpi_J(\bp')=1$ and such that $B_J$ is a minimal parabolic subgroup
containing $B_J\cap\lexp {\bp'} B_J$.
\end{enumerate} 
\end{proposition}
\begin{proof}
(1) is clear from Proposition~\ref{prepdbcoset}(1).

Let us prove (2).
Let $\bj\in
B_J\cap  B_I^{\bb}$ and  $\bi\in B_I$  be such that  $\bb \bj  = \bi\bb$.
Since $\pr(\bb)$ is $I$-reduced-$J$ (see Proposition~\ref{prepdbcoset}(1)) we
 can apply Proposition~\ref{fourre-tout}(\ref{ft4}), which gives $\bpi_I(\bb  \bj)  =
\bpi_I(\bb)\varphi_{\pr(\bb)}(\bpi_{J_1}(\bj)) =
\varphi_{\pr(\bb)}(\bpi_{J_1}(\bj))$.  On the other hand, $\bpi_I(\bi\bb) =
\bi\bpi_I(\bb)  = \bi$. So $\bi\in B_{I_1}$.

Similarly we may apply Proposition~\ref{counterpart} to get
 $\bj=\pi_J^r(\bb\bj)=\pi_J^r(\bi\bb)=\varphi_{\pr(\bb)}\inv(\pi_{I_1}^r(\bi))$,
which proves that $\bj\in B_{J_1}$.

We prove now (3).
We first prove that $\pi_I(\bp)=1$.
By Proposition~\ref{fourre-tout}(\ref{ft2}) we have $1=\bpi_I(\bb)=\bpi_I(\bp)\bpi_I(\bw)$. By
Proposition~\ref{fourre-tout}(\ref{ft5}) we have 
$\bpi_I(\bw)=1$, hence $\bpi_I(\bp)=1$.

Now we have $B_I\cap\lexp \bb B_J=B_I\cap \lexp{\bp\bw} B_J=B_I\cap \lexp\bp B_I$, the last
equality by Lemma~\ref{ribbon in B}.
If $\bi$ is in $B_I\cap \lexp\bp B_I$ we have $\bp\bi=\bi'\bp$ for some $\bi'\in B_I$. Hence
$\pi_I(\bp\bi)=\pi_I(\bp)\pi_I(\bi)=\pi_I(\bi)=\bi$ and
$\pi_I(\bi'\bp)=\bi'\pi_I(\bp)=\bi'$, so that $\bi=\bi'$ is in $C_{B_I}(\bp)$.

We prove (4). Let $\lexp\bi B_J$ with $J\subseteq I$ and $\bi\in B_I$ be any
parabolic subgroup of $B_I$ containing $B_I\cap \lexp \bp B_I$.
Then $\lexp{\bp\bi}B_J$ also contains $B_I\cap \lexp \bp B_I$ since this
intersection is centralised by $\bp$. 
We have thus 
$\lexp\bi B_J\cap\lexp{\bp\bi}B_J\supseteq  B_I\cap\lexp \bp B_I$, hence
equality $\lexp\bi B_J\cap\lexp{\bp\bi}B_J= B_I\cap\lexp \bp B_I$,
that is $B_I\cap\lexp \bp B_I$ is conjugate in $B_I$ to $B_J\cap
\lexp{\bp'}B_J$ where $\bp'=\bi\inv\bp\bi$. We have
$\bpi_I(\bp')=\bi\inv\bpi_I(\bp)\bi=1$. If we take $\lexp\bi B_J$ minimal
containing $B_I\cap\lexp \bp B_I$ we get~(4); such a parabolic exists since $I$ is finite.
\end{proof}
Note that if we assume that the intersection of two parabolic subgroups is a
parabolic subgroup, then in (4) above $B_J=B_J\cap\lexp{\bp'}B_J$ and by
the proof of (3) we have $\bp'\in C_P(B_J)$.

\section{Minimal length  elements in a conjugacy class}\label{conjectures}
In this section, we first recall a result on elements of $B^+$ 
(Proposition~\ref{pure  centralizer})  which  is  an  important  step in proving that in
spherical type Artin  groups  there  is  a  unique  minimal  parabolic subgroup
containing  a given  element (see \cite[Introduction of section 6]{CumplidoetAll}).
We  then show  that a  conjecture (Conjecture~\ref{conjecture2})  
which generalises Proposition~\ref{pure centralizer} to
all conjugacy classes is equivalent in all Artin groups to the existence of
a unique minimal parabolic subgroup containing a given element. 
We conclude with a partial result (Proposition~\ref{conjugaison V2})
on conjugacy. 

\subsection*{Positive elements}
The  goal of this  subsection is to  introduce the concepts  and results on
elements  of $B^+$  needed, ending  with a  proof of  Proposition~\ref{pure
centralizer}.

For  $I\subseteq S$ and $\bb\in B^+$ we define $H_I(\bb)$ \index{HI@$H_I$}
as  the longest prefix of $\bb$ in  $B^+_I$; we say that an element $\bg\in
B^+$  is $I$-reduced \index{Ired@$I$-reduced} if  $H_I(\bg)=1$. We say that
an  element~$\bg\in  B^+$  is  an  $I$-ribbon  if  it  is  $I$-reduced and
$\bI^\bg\subseteq  B^+$. Since $\lZ$ is  constant on a conjugacy class, this
is seen to be equivalent to $\bI^\bg\subseteq\bS$. For $\bb\in B^+$ we define
the  support  $\supp(\bb)$  \index{supp@$\supp$}  as  the  smallest  subset
$I\subseteq S$ such that $\bb\in B_I^+$. This  is well defined since the two
sides  of a  braid relation  involve the  same elements  of $\bS$. Note that
the notions  of $I$-reduced element and of  $I$-ribbon in $B^+$ are
coherent   with  the   corresponding  definitions  in  $W$  introduced  in
Section~\ref{section1}:   the  element   $\bw$  in   $\bW$  is  $I$-reduced
(\emph{resp.}  a $I$-ribbon) if and only if $\pr(\bw)$ is. This is not true
if $\bw\notin\bW$. For instance  $\bs^2$ is not  $\{s\}$-reduced but 
$\pr(\bs)$ is, and  is  even  an $\{s\}$-ribbon.  If  $\bs$  and $\bt$ do
not commute, then
$\bs^2\bt$ is $\{t\}$-reduced but $\pr(\bs^2\bt)= t$ is not.

Recall that the Garside normal form has been introduced in Section~\ref{section2}. For  $\bb\in B^+$  we denote by $\head(\bb)$ \index{head@$\head$} the first term of the
Garside   normal   form   of~$\bb$   and  define
$\tail(\bb)=\head(\bb)\inv\bb$.
\index{tail@$\tail$}
\begin{lemma}\label{divers}\phantom{a}
\begin{enumerate}
\item For $\bb,\bg\in B^+$, if $\bb^\bg\in B^+$ then $\bb^{\head(\bg)}\in B^+$
\item If $\bg\in B^+$ is an $I$-ribbon, then $\head(\bg)$ is an $I$-ribbon.
\item If $\bg\in B^+$ is an $I$-ribbon, then $\pr(\bg)$ is an $I$-ribbon.
\end{enumerate}
\end{lemma}
\begin{proof}
For  (1), the assumption  $\bb^\bg\in B^+$ is equivalent to $\bg$
left-dividing
$\bb\bg$  from  which  it  follows  that $\head(\bg)$ left-divides
$\head(\bb\bg)=\head(\bb
\head(\bg))$,   in   particular   $\head(\bg)$ left-divides $   \bb  \head(\bg)$,  that  is
$\bb^{\head(\bg)}\in B^+$.

 For (2), by (1) we get $\bI^{\head(\bg)}\subseteq B^+$. And clearly if $\bg$ is
 $I$-reduced then $\head(\bg)$ also.

Finally  (3)  is  clear if  $\bg\in\bW$  and  the  general case follows from
(2) by induction on the number of terms in the Garside normal form of $\bg$.
\end{proof}
\begin{lemma}\label{$I$-head and $I$-tail preserved by ribbons}
Let $\bg\in B^+$ be an $I$-ribbon and
let $\bh\in B^+$. Then $H_I(\bg\bh)^\bg=H_{I^{\pr(\bg)}}(\bh)$
and $H_I(\bg\bh)\inv\bg\bh=\bg H_{I^{\pr(\bg)}}(\bh)\inv\bh$.
\end{lemma}
\begin{proof} Let $\bs\in \bI$ and set $\bs'=\bs^\bg$. Both formulae clearly
follow if we show that it is equivalent that $\bs$ left-divides $ \bg\bh$ or that
$\bs'$ left-divides $ \bh$.

Now from $\bs \bg= \bg\bs'$ it follows that the right-lcm of $\bs$ and $\bg$
divides $\bs \bg$, thus it must be equal to $\bs \bg$, because 
$\lZ(\rightlcm(\bs,\bg))$ is
at least $\lZ(\bg)+1=\lZ(\bs\bg)$.
Thus $\bs$ left-dividing $ \bg\bh$ is equivalent to $\bs
\bg$ left-dividing $ \bg\bh$ or equivalently to $\bg\bs'$ left-dividing $ \bg\bh$, 
which is finally equivalent to $\bs'$ left-dividing $ \bh$.
\end{proof}
\begin{proposition}\label{Positive conjugate implies ribbon}
Let $\bg,\bb\in B^+$ be such that $\bg$ is $\supp(\bb)$-reduced
and such that $\bb^\bg\in B^+$. 
Then $\bg$ is  a $\supp(\bb)$-ribbon and we have $\supp(\bb^\bg)=\supp(\bb)^{\pr(\bg)}$.
\end{proposition}
\begin{proof}
We first show that we can reduce to the case where $\bg$ is simple, by arguing
by induction on the number of terms of the Garside normal form of $\bg$.

By Lemma~\ref{divers}(1) we have $\bb^{\head(\bg)}\in B^+$; 
and for $\bg$  to  be  $\supp(\bb)$-reduced,  we  certainly  need  $\head(\bg)$
to be $\supp(\bb)$-reduced. 

If we know the theorem in the case where $\bg$ is simple,
it    follows    that the set $\bJ=\bI^{\head(\bg)}$, where $I=\supp(\bb)$,
is a subset of $\bS$;    and    thus
$\bb^{\head(\bg)}$  has  support  $J$.  We  also  have  that  $\tail(\bg)$ is
 $J$-reduced by Lemma~\ref{$I$-head and $I$-tail preserved by ribbons}
 applied with $\head(\bg)$ for $\bg$ and $\tail(\bg)$ for $\bh$. Since
$(\bb^{\head(\bg)})^{\tail(\bg)}\in B^+$,  we conclude by induction on the number
of  terms  of  the  normal  form of $\bg$ that $\bJ^{\tail(\bg)}\subseteq \bS$,
 whence the result since $\bJ^{\tail(\bg)}= \bI^\bg$.

We now show the theorem when $\bg$ is simple. We write the condition
$\bb^\bg\in B^+$ as $\bb\bg=\bg\bu$ with $\bu\in B^+$.

 We proceed by an induction on $\lZ(\bb)$. For $\bs\in \bS$ dividing $\bb$,
we write $\bb=\bs \bb'$. We have that $\bs$ left-divides $\bb\bg=\bg\bu$.

To go on we need the following lemma.
\begin{lemma}Let  $\bg\in B^+$ be simple and let $\bs\in \bS$, $\bu\in B^+$
be  such that $\bs$ does not divide $\bg$  but $\bs$ divides $\bg\bu$. Then
we   have   $\bu=\bu'\bt   \bu''$   where   $\bt\in  \bS$, $\bu',\bu''\in B^+$
are such  that  $\bs \bg\bu'=\bg\bu'\bt$ is simple.
\end{lemma} 
\begin{proof}
First  note  that  $\bs$ left-divides $\head(\bg \bu)$. Let $\bg\bu_1=\head(\bg\bu)$.
Since in $W$ we have $\ell_S(sgu_1)<\ell_S(gu_1)$ and $\ell_S(sg)>\ell_S(g)$   (in other words the reflection  $s$ is in the left descent  set of $gu_1$ but
not  in that  of $g$), we have  by the exchange lemma $gu_1=gu't u_2$ where
$t\in  S$ and  $sgu'=gu't$. Lifting  back to  $B^+$ we  get the  lemma with
$\bu''= \bu_2\tail(\bg\bu)$. \end{proof}

We now resume the proof of the proposition.
Since $\bs$ does not divide $\bg$, the lemma gives $\bu=\bu'\bt \bu''$ with
$\bt\in \bS$ and $\bs \bg\bu'=\bg\bu'\bt\in\bW$. From $\bb\bg=\bs
\bb'\bg=\bg\bu=\bg\bu'\bt \bu''=\bs \bg\bu'\bu''$ we deduce
$\bb'\bg=\bg\bu'\bu''$.  Thus  $\bb',\bg$  satisfy  the  assumptions of the
 theorem, thus by induction on $\lZ(\bb)$ we have
 $\bK^\bg\subseteq \bS$ where $K=\supp(\bb')$. We still have to prove that $\bs^\bg\in \bS$.
This is already proven unless $\bs\not\in\bK$, which we assume now.

We  write $\bb'\bg=\bg\bb''$ for
some  $\bb''\in B$.  Since $\bs  \bg\bb''=\bg\bu=\bg\bu'\bt \bu''=\bs \bg
\bu'  \bu''$, we have $\bb''=\bu'\bu''\in B^+$. Now $\bg\by\bg\inv\in
\bK$, for any $y\in\supp(\bb'')$ hence the element $\bv=\bg\bu'\bg\inv$ is in  $B^+_K$. From $\bs
\bg\bu'=\bg\bu'\bt  $  we  get  $\bs  \bv  \bg=\bv\bg\bt\in\bW$, which we write as
$(\bv\inv  \bs  \bv)\bg=\bg\bt$.  Since  $\bg$  is $\supp(\bb)$-reduced, its
image  in $W$ is also and in $W$ we have $1+\ell_S(g)=\ell_S(gt)=\ell_S(v\inv s vg)=\ell_S(v\inv s
 v)+\ell_S(g)$, the last equality since $v\inv sv$ is in $W_{\supp(\bb)}$, 
whence  $\ell_S(v\inv s  v)=1$. Now  $\bv$ and  $\bs\bv$ are simple since 
$\bs\bv\bg=\bs\bg\bu'$, which is simple. Thus $\bv$ divides $\bs\bv$ because $l(v)+l(v\inv sv)=l(sv)$ which implies that 
$\bv\bx=\bs\bv$ where $\bx\in\bS$ lifts $v\inv sv\in S$. 
It follows that  $\bv\inv\bs\bv\in\bS$.  But  $\supp(\bv)\subset\bK$  so  that
$\bs\notin\supp   (\bv)$,  hence  $\bv\inv   \bs  \bv\in\bS$  implies
$\bv\inv \bs \bv=\bs$, thus $\bs^\bg=\bt$.
\end{proof}
\begin{proposition}\label{pure centralizer}Let $I\subseteq S$ and let
 $\bp\in B^+$ be a pure element such that $H_I(\bp)=1$; if $\bp$
commutes with an element in $B^+$ of support $I$
then $\bp$ centralises $B_I$.
\end{proposition}
\begin{proof}
Let $\bi\in B^+$ have support $I$.
By Proposition~\ref{Positive conjugate implies ribbon} conjugating by $\bp$ maps
$\bI$ into $\bS$. Since
$\bp$ is pure applying $\pr$ to the equality $\bp\bs=\bs'\bp$ with $\bs\in\bI,
\bs'\in\bS$ we get $\pr(\bs)=\pr(\bs')$ hence $\bs=\bs'$.
\end{proof}
\subsection*{General elements}
Recall that in the introduction we have defined $\lS:B\to\BN$ to be the length
function  on  $B$ with respect to
the generating set $\bS^{\pm 1}$. For $\bb\in B$, it is equal to
$\hl(b)$, the length of a minimal word $b\in(\bS^{\pm1})^*$ representing $\bb$.

\begin{definition} For a word $b\in(\bS^{\pm1})^*$, we call $\supp(b)$ the
minimal subset $I\subseteq S$ such that $b\in(\bI^{\pm1})^*$.
\end{definition}

The  following proposition, which enables the definition of support for all
elements, is the convexity theorem of Charney and Paris
\cite[Theorem~1.2]{Charney-Paris}.
\begin{proposition}\label{convexity} For $\bb\in B$, the set $\supp(b)$
is independent of the
word $b$ of shortest length representing $\bb$.
We call it support of $\bb$, denoted by $\supp(\bb)$. For any
word $b'$ representing $\bb$ we have $\supp(b')\supseteq\supp(\bb)$.
\end{proposition}
Note that the definition of support for elements of $B$ is compatible with
that given above Lemma~\ref{divers} for
elements of $B^+$.
\begin{proof}
Let $b$ be a word  of minimal
length  for $\bb$  and let $J$ be the support of a word $b'$ for $\bb$.
Since $\hpi_J(b)$ represents $\bb$, by minimality we have $\hl(\hpi_J(b))=\hl(b)$, 
whence by Remark~\ref{obvious}(5), $\supp(b)\subseteq J$,  which proves the
proposition.
\end{proof}

We denote by $\conj(\bb)$ \index{conj@$\conj$} the conjugacy class of $\bb$
in $B$, or $\conj_B(\bb)$ when we want to specify $B$.
\begin{proposition}\label{bg=gh}
Let  $\bi,\bj\in B$ be conjugate elements  with respective supports $I$ and
$J$. Assume $\lS(\bj)$ minimal in $\conj(\bi)$. Let $\bg\in B$ be such that
$\bi\bg=\bg  \bj$;  then  $\pr(\bt_I(\bg))$  is  a $I'$-ribbon-$J$ for some
$I'\subseteq I$. If in addition $\lS(\bi)$ is minimal in $\conj(\bi)$, then
$I'=I$.
\end{proposition}
\begin{proof}
We   have  $\bpi_I(\bi\bg)=\bi\bpi_I(\bg)=\bpi_I(\bg\bj)$.  By  
Proposition~\ref{pi(bb')}  we  have  $\bpi_I(\bg\bj)=\bpi_I(\bg)\bi'$  for some element
$\bi'\in B_I$ such that given a minimal word $j$ for $\bj$, the word
$i'=p^{*-1}(\vN(\vN(\lexp{t_I(\pr(\bg))}p^*(j))\cap\Phi_I))$ is  a
representative  for~$\bi'$.  But $\lS(\bi') \leq  \hl(
p^{*-1}(\vN(\vN(\lexp{t_I(\pr(\bg))}p^*(j))\cap\Phi_I)))\leq\hl(j)=
\lS(\bj) \leq \ell_{\bS^{\pm1}}(\bi')$, where the last inequality comes from
the minimality of $\lS(\bj)$ in its conjugacy class, since 
$\bi'$  is  conjugate  to  $\bj$.  Thus we have $\lS(\bi')=\lS(\bj) = \hl(
p^{*-1}(\vN(\vN(\lexp{t_I(\pr(\bg))}p^*(j))\cap\Phi_I)))$.   This  forces
all terms of  $\vN(\lexp{t_I(\pr(\bg))}p^*(j))$ to be in $\Phi_I$.
Using that $\vN$ is an involution, we get
$i'=p^{*-1}(\vN(\vN(\lexp{t_I(\pr(\bg))}p^*(j))))=
p^{*-1}(\lexp{t_I(\pr(\bg))}p^*(j))$.  But we know that in the formula for
$\hpi_I$  the roots  to which  one applies  $p^{*-1}$ are in $\Pi_I^{\pm1}$. Hence
$t_I(\pr(\bg))=\pr(\bt_I(\bg))$  is  an  $I'$-ribbon-$J$  where $I'$ is the
support of $i'$ (and $\bi'=\varphi_{t_I(\pr(\bg))}(\bj$)).

If $\lS(\bi)$ is minimal  in the conjugacy class of $\bi$, we may apply
the previous result to the equality $\bi'\bpi_I(\bg)\inv=\bpi_I(\bg)\inv\bi$
to conclude that $t_{I'}(\pr(\bpi_I(\bg)\inv))$ is an $I''$-ribbon-$I$ for some
$I''\subseteq I'$. But this implies that $t_{I'}(\pr(\bpi_I(\bg)\inv))$ is
reduced-$I$; since this element is in $W_I$, this implies
$t_{I'}(\pr(\bpi_I(\bg)\inv))=1$ thus $I''=I=I'$.
\end{proof}

\begin{corollary} \label{corollary}
Let $\bi,\bj\in  B$ be conjugate elements of respective supports $I,J$;
assume that $\lS(\bi)$ and $\lS(\bj)$ are equal and minimal in $\conj(\bi)$. Then
\begin{enumerate}
\item $I$ and $J$  are $W$-conjugate.
\item For $\bg$ such that $\bi\bg=\bg\bj$, 
the element $\bw\in\bW$ defined by $\pr(\bw)=\pr(\bt_I(\bg))$ is an 
$I$-ribbon-$J$ such that $\bi^{\bpi_I(\bg)\bw}=\bj$.
\end{enumerate}
\end{corollary}
\begin{proof}
(1) comes (2). (2) comes from the facts in the proof of
 Proposition~\ref{bg=gh} that $I'=I$ and $\bi^{\bpi_I(\bg)}=\bi'=\varphi_{t_I(\pr(\bg))}(\bj)=
 \varphi_{\pr(\bw)}(\bj)$.
\end{proof}
\begin{proposition} \label{conj min V2}
For $I\subseteq S$ let $\bi\in B_I$. Let $\bg\in B$ be such that 
$\lS(\bg\inv\bi\bg)$ is minimal in $\conj_B(\bi)$.
Then $\bpi_I(\bg)\inv\bi\bpi_I(\bg)\in\conj_{B_I}(\bi)$ has $\lS$ minimal
in $\conj_B(\bi)$.
\end{proposition}
\begin{proof}
This is a consequence of the proof of Proposition \ref{bg=gh} where
it is shown that the element $\bi'=\bi^{\bpi_I(\bg)}$ is has $\lS$ minimal in $\conj_B(\bi)$.
\end{proof}
From Proposition~\ref{conj min V2},  we immediately deduce
\begin{corollary}\label{min in conj} For $I\subseteq S$ and $\bb\in B$, 
either $\conj_B(\bb)\cap B_I$ is empty or it contains an element~$\bi$
such that $\lS(\bi)$ is minimal in $\conj_B(\bb)$. 
\end{corollary}
\begin{proposition}\label{minimal  support}  Let  $\bb\in  B$  be such that
$\lS(\bb)$  is minimal in $\conj(\bb)$.  Then $B_{\supp(\bb)}$ is a minimal
parabolic subgroup containing $\bb$, and any minimal parabolic subgroup
containing $\bb$ is of the form $\lexp{\bp}B_{\supp(\bb)}$ for some element
$\bp$ in $C_P(\bb)$ such that $\bpi_{\supp(\bb)}(\bp)=1$.
\end{proposition}
\begin{proof}
We first observe that any minimal parabolic subgroup containing $\bb$ is of the form
$\lexp\bg  B_I$ with $\lS(\bb^\bg)=\lS(\bb)$ and $\supp(\bb^\bg)=I$. Indeed
if  $\lexp\bg B_I$ is minimal containing $\bb$ then we apply by Proposition~\ref{conj min
V2} with $\bg\inv$ for $\bg$ and $\bb^\bg$ for $\bi$. This shows that
up  to  replacing  $\bg$  by  $\bg\bpi_I(\bg\inv)$ we may assume that
$\lS(\bb^\bg)=\lS(\bb)$; and we certainly have $\supp(\bb^\bg)=I$ otherwise
the  proper parabolic subgroup $B_{\supp(\bb^\bg)}$  of $B_I$ would contain
$\bb^\bg$ and $\lexp\bg B_I$ would not be minimal containing $\bb$.

Let now $\lexp\bg B_I$ be  any minimal parabolic subgroup containing $\bb$.
We may assume that $\lS(\bb^\bg)=\lS(\bb)$ and that $\supp(\bb^\bg)=I$. 
It follows by Corollary~\ref{corollary}(1) that
$\supp(\bb)$ is conjugate to $I$. So, up to changing $\bg$, we could have started with a parabolic subgroup $\lexp\bg
B_I$ where $I=\supp(\bb)=\supp(\bb^\bg)$. It also follows from Corollary~\ref{corollary}(2) that $\bb^\bg=
\bb^{\bpi_I(\bg)\bw}$ where $\bw\in\bW$ is an $I$-ribbon-$I$ and
 where $\bp=\bg\bw\inv\bpi_I(\bg)\inv$ is pure.
In particular $B_I^{\bpi_I(\bg)\bw}=B_I$ so $\lexp\bg
B_I=\lexp{\bp}B_I$. We have
$\bb\in B_I\cap \lexp\bp B_I=\lexp{\bpi_I(\bp)}(B_I\cap
\lexp{\bt_I(\bp)} B_I)=\lexp{\bpi_I(\bp)}
C_{B_I}(\bt_I(\bp))=C_{B_I}(\bp\bpi_I(\bp)\inv)$, the second equality by
Proposition~\ref{propIntro4}, thus $\bb$ commutes to
$\bp\bpi_I(\bp)\inv$. Now
$\lexp{\bp\bpi_I(\bp)\inv}B_I=\lexp\bp B_I$.  Replacing $\bp$
 by $\bp\bpi_I(\bp)\inv$ we get that any minimal parabolic subgroup containing $\bb$ is as
 in the statement. 

To show that these parabolic subgroups are actually minimal
it is sufficient to show that $B_{\supp(\bb)}$ is a minimal parabolic subgroup
containing $\bb$.
By what we have just seen, any minimal parabolic subgroup of $B_{\supp(\bb)}$ containing $\bb$ is of the
form $\lexp \bp B_{\supp(\bb)}$ with $\bp\in P_{\supp(\bb)}$ and $\bpi_{\supp(\bb)}(\bp)=1$, so $\bp=1$.
\end{proof}
Since minimality of parabolic subgroups transfers by conjugation,
Proposition~\ref{minimal support} proves the existence of a minimal
parabolic subgroup containing any given element.

If the intersection of two
parabolic subgroups is parabolic, there exists a unique minimal parabolic subgroup
containing an element.
Thus the validity of the following conjecture is supported by the results of
\cite{AntolinFoniqi,CumplidoetAll3,CumplidoetAll,CumplidoetAll2,Morris-Wright}
\begin{conjecture}\label{conjecture1}
There exists a unique minimal parabolic subgroup containing a given $\bb\in B$.
\end{conjecture}
The following conjecture generalises Proposition~\ref{pure centralizer},
replacing $H_I$ by $\bpi_I$. Note that for a positive element, $H_I$ is a
prefix of $\bpi_I$.
\begin{conjecture}\label{conjecture2}
Let $\bb\in B$ such that $\lS(\bb)$  is minimal in $\conj(\bb)$.
Then any $\bp\in P$ which centralises $\bb$ and is such that
$\bpi_{\supp(\bb)}(\bp)=1$ centralises $B_{\supp(\bb)}$.
\end{conjecture}
\begin{remark}
The   $\lS$-minimality  assumption   in  Conjecture~\ref{conjecture2}  is
necessary.  In the  braid group  of $\mathfrak  S_4$ with $\bS=\{\bs,\bt,\bu\}$
take $\bb=\bs^\bt$, $I=\supp(\bb)=\{\bs,\bt\}$, and $\bp=(\bu^2)^\bt$. Then
$\bpi_I(\bp)=\bpi^r_I(\bp)=1$   and  $\bb^\bp=\bb$   but  $\bp$   does  not
centralise $\bs$ or $\bt$.
\end{remark}

\begin{proposition} Conjectures~\ref{conjecture1} and~\ref{conjecture2}
are equivalent.
\end{proposition}
\begin{proof}
Assuming Conjecture~\ref{conjecture1}, 
if $\bb$ and $\bp$ are as in~\ref{conjecture2}, then by 
Proposition~\ref{minimal support} $B_I$ and  $\lexp\bp B_I$, where
$I=\supp(\bb)$, are two minimal parabolic subgroups containing $\bb$,
hence they are equal
and by Proposition~\ref{propIntro4} $\bp$ centralises $B_I$, 
whence Conjecture~\ref{conjecture2}.

We  prove  the  converse.  Up  to  conjugating  $\bb$, we can assume that
$\lS(\bb)$  is  minimal in $\conj_B(\bb)$.
Then, by  Proposition~\ref{minimal  support},  the
minimal  parabolic subgroups containing $\bb$ are the $\lexp\bp B_{\supp(\bb)}$ for $\bp$
in $C_P(\bb)$ such that $\bpi_{\supp(\bb)}(\bp)=1$.
By  Conjecture~\ref{conjecture2},  $\bp$ centralises
$B_{\supp(\bb)}$. We deduce that $B_{\supp(\bb)}=\lexp\bp B_{\supp(\bb)}$ is the
only minimal parabolic subgroup containing $\bb$.
\end{proof}

Given  a Coxeter system~$(W,S)$ and a  pair $(I,J)$ with $I,J \subseteq S$,
we define the ribbon isomorphism problem in $W$ for $(I,J)$ to be the
following problem: determine all the bijective maps $\varphi : J\to I$ that
extend to an isomorphism $\varphi_w : W_J \to W_I$ such that $w$ is a
$I$-ribbon-$J$ in $W$.

\begin{proposition}\label{conjugaison V2} 
Assume $\Lambda$ is a non-empty set of subsets of $S$ such that for any $I$ in
$\Lambda$,  the word problem and the  conjugacy problem are solvable in $B_I$ and the ribbon isomorphism problem is solvable in $W$ for any pair $(I,J)$ in $\Lambda$.
Then the two following problems are solvable.
\begin{enumerate}
\item Given $I$ in $\Lambda$ and a word $b$ on $\bI^{\pm 1}$, decide whether the
word $b$ represents the unity in $B$.
\item  Given two finite subsets $I, J$ in $\Lambda$ and words $i, j$ on $\bI^{\pm 1}$
and $\bJ^{\pm 1}$, respectively, decide whether or not $i$ and $j$ represent
conjugate elements in $B$.
\end{enumerate}
Moreover,  if the solution to the conjugacy problem in each parabolic subgroup
$B_I$ with $I\in \Lambda$ provides a conjugating element, then the solution of
problem (2) provides a conjugating element in $B$ too.
\end{proposition}
\begin{proof}
Solution to problem (1) is clear as $B_I$ embeds in $B$ and we have a
solution to the word problem in $B_I$. Given $i$ and $j$ as in (2),
denote  by $\bi$  and $\bj$  the elements  represented by $i$ and $j$,
respectively. By the solution to the conjugacy problem in the parabolic subgroup
$B_I$, by testing for conjugacy the words in $(\bI^{\pm 1})^*$ of length less than or
equal to that of a word for $\bi$, we can find a representative word on
$\bI^{\pm 1}$ for each element $\bi'\in B_I$ that is conjugate to $\bi$ and
such that $\ell_{\bI^{\pm1}}(\bi')$ is minimal in $\conj_{B_I}(\bi)$.
By Proposition~\ref{convexity},
$\ell_{\bI^{\pm1}}(\bi')  = \lS(\bi')$ and  by Proposition~\ref{conj min
 V2},  $\lS(\bi')$ is also minimal in $\conj_B(\bi)$.
Similarly, we can find a representative word on $\bJ^{\pm 1}$ for an
element $\bj'\in B_J$ that is conjugate to $\bj$ and such that
$\ell_{\bJ^{\pm1}}(\bj')$ is minimal in $\conj_{B_J}(\bj)$.
Again, $\lS(\bj')$  is minimal  in $\conj_B(\bj)$. 
Now $\bi$ and $\bj$ are conjugate if and only if some $\bi'$
and $\bj'$ are. By  the proof of  Proposition~\ref{conj min V2}, some
$\bi'$ is conjugate to $\bj'$ if and only if they are conjugate 
by the lift to $\bW$ of some $I'$-ribbon-$J'$ in $W$ where $J'$ is the support of
$\bj'$ and $I'$ is the support of $\bi'$.  Since the ribbon isomorphism problem is solvable in $W$ for any pair $(I,J)$ in $\Lambda$, we are done.
\end{proof}

When $S$ is finite the  hypothesis concerning  the ribbon isomorphism problem  is always satisfied as stated in  the following result.  As a consequence, we get Proposition~\ref{propIntro6}.

\begin{proposition}\label{determiner aut ruban}  Assume $(W,S)$ is a Coxeter system with $S$ finite.  The ribbon isomorphism problem is solvable  in $W$ for any pair $(I,J)$ with $I,J\subseteq S$.  
\end{proposition}
\begin{proof} Let $I,J\subseteq S$. 
Let $\Lambda$ denote the set of subsets $I'\subseteq S$ whose Coxeter graph is
isomorphic to the Coxeter graph of $I$, and
let $s$ be the number of graph automorphisms of the Coxeter graph of $I$. 
Since $S$ and $I$ are finite it takes finite time to compute $\Lambda$ and $s$.
Any one-to-one map $\varphi : I'\to I$ that extends to an
isomorphism $\varphi : W_{I'} \to W_I$ has to induce a isomorphism of graphs.
So the set $\tilde{\Lambda}$ of pairs $(I',\varphi)$ such that $\varphi :
I'\to I$ extends to an isomorphism $\varphi : W_{I'} \to W_I$ is finite and
its cardinality is bounded by $|\Lambda|\times s$. 
On the other hand, every ribbon can be decomposed into a finite product of
elementary ribbons $(I',v[s,I''],I'')$  (see \cite[Proposition~2.3]{Brink-Howlett} and the set of
elementary ribbons is finite (its cardinality is at most $|S-I|\times
|\Lambda|$); the set of associated isomorphisms is easy to determine. 
Now, consider the finite oriented labelled graph with vertex set
$\tilde{\Lambda}$ and such that there is an edge from $(I',\varphi')$ to $(I'',\varphi'')$ labelled by the elementary
$I'$-ribbon-$I''$ $\tau$ if $\varphi''  = \varphi'\circ\varphi_\tau$.
Then  any $I$-ribbon  can be  read as  a path  in this  graph based at $(I,
\Id)$,   and  any  path  without  loop  has   a  length  of  at  most
$|\tilde{\Lambda}|$,  which  is  less  than  $|\Lambda|\times  s$. So,
computing  all the isomorphisms that correspond to  a path whose length is no
more  than $|\Lambda|\times s$ provides an answer to the ribbon isomorphism
problem.
\end{proof}

\section{A new topological version of the retraction}\label{section4}
In this section we assume that $S$ is finite, since we use the results of
\cite{VanderLek}. We did not check whether Van der Lek's results could be extended
to infinite $S$.
According  to Lemma~\ref{piw=1}, it is  sufficient to define $\bpi_I$ on the
pure  braid group and on simple braids to define it everywhere. We now give
a topological definition of $\bpi_I$ restricted to $\pr\inv(W_I)$.

Let $V$ be the reflection representation of $W$, and $V_\BC$ be the
complexified space. The set
$T$ is in bijection with the reflecting hyperplanes of $W$. Let $H_s$ be
the hyperplane corresponding to $s\in T$. We define
$X=V-\bigcup_{s\in T}H_s$ and similarly $X_I=V-\bigcup_{s\in T_I}H_s$.
Any element of $V_\BC$ is of the form $v=v_1+\sqrt{-1}v_2$ where $v_1,v_2\in V$.
We write $v_1=\Re(v)$ and $v_2=\Im(v)$.
Now, following Van der Lek, let $\CC$ be the open Tits cone in $V$ and let
$Y=\{v\in X\mid \Im(v)\in\CC\}$ (\cite[II, (3.2)]{VanderLek}). We define similarly $Y_I$ (note that
though the Tits cone of $W_I$ is bigger than $\CC$ its intersection with
$X_I$ is the same).
Choosing some $y_0\in Y$,
the pure braid group is $P=\Pi_1(Y,y_0)$ and similarly the pure braid group of
$B_I$ is $P_I=\Pi_1(Y_I,y_0)$. There is a natural morphism $P\to P_I$
(since $Y\subseteq Y_I$). We will show that this morphism is
equal to $\bpi_I$. More generally, we have
\begin{lemma} $\Pi_1(Y/W_I,\overline y_0)\simeq \pr\inv(W_I)$
\end{lemma}
\begin{proof}
We have the commutative diagram
$$\begin{CD}
 1 @>>> P @>>> B=\Pi_1(Y/W)@>\pr>> W@>>>1\\
@.@|@AAA@AAA\\
 1@>>>P@>>>\Pi_1(Y/W_I,\overline y_0)@>>>W_I@>>>1\\
\end{CD}$$
where each line corresponds to a covering. The vertical map in the middle
is induced by the natural map $Y/W_I\to Y/W$;
it is injective since the left one is an
isomorphism. This diagram proves the lemma.
\end{proof}
\begin{proposition} The natural morphism 
$\Pi_1(Y/W_I,\overline y_0)\to\Pi_1(Y_I/W_I,\overline y_0)$ coincides with
the restriction of $\bpi_I$ to $\pr\inv(W_I)$.
\end{proposition}
\begin{proof}
We prove that the two morphisms coincide on generators of $\pr\inv(W_I)$. By
Lemma~\ref{generators} this group is
generated by $B_I$ and the elements $\bw\bs^2\bw\inv$ with $\bs\in\bS$ and 
$\bw\bs\in\bW$ an $I$-reduced element.
Let $C$ be the fundamental chamber in $V$, that is the subset defined
by $\alpha_s(x)>0$ for all $s\in S$ where $\alpha_s$ is a linear form
such that  $H_s=\{x\in V_\BC\mid \alpha_s(x)=0\}$. The chambers are the images
of $C$ under $W$. Let $F_I$ be the facet of $C$ defined by $\alpha_s(x)=0$
for $s\in I$ and $\alpha_s(x)>0$ for $s\not\in I$. As in \cite[II, Definition 4.8]{VanderLek}, we
define $Y(I)$ as the subset of $Y$ formed of elements whose imaginary part is
in $F_I$ or in a chamber whose closure contains $F_I$.
Then as a subgroup of $\Pi_1(Y/W_I,\overline y_0)$, the group $B_I$
is the image of $\Pi_1(Y(I)/W_I,\overline y_0)$, and the natural
morphism to $\Pi_1(Y_I/W_I,\overline y_0)$ is an isomorphism since Van der Lek
\cite[II, Proposition 4.9]{VanderLek} shows that the inclusion induces a homotopy equivalence between $Y(I)$ and
$Y_I$.

For the elements $\bw\bs^2\bw\inv$, we need to show that their image
in $\Pi_1(Y_I/W_I,\overline y_0)$ is
trivial. But we claim that the image is
already trivial in the pure braid group $\Pi_1(Y_I,y_0)$.
This can be seen from the description in Van der Lek 
\cite[I, Theorem~2.21 and II, Discussion 3.10]{VanderLek} of such paths by
galleries. 
The element $\bs^2$ corresponds to the gallery $C,sC,C$. If $s\not\in I$ the
chambers $C$ and $sC$ are in the same chamber of $Y_I$ so that the gallery
becomes trivial in $Y_I$. By conjugation, the same happens to~$\bw\bs^2\bw\inv$.
\end{proof}
\bibliographystyle{acm}
\bibliography{biblio_DMG}
\printindex
\end{document}